\documentclass[10pt,reqno]{amsart}
\usepackage{amsmath}
\usepackage{amssymb}
\usepackage{amsthm}
\usepackage{eepic,epic}
\usepackage{graphics,psfrag,graphicx,subfigure,comment,enumerate,tcolorbox}
\usepackage{xcolor}
\usepackage{tikz}
\usepackage{mathtools}

\newtheorem{thm}{Theorem}[section]

\newtheorem{lem}[thm]{Lemma}

\theoremstyle{definition}

\theoremstyle{remark}
\newtheorem{rem}[thm]{Remark}
\numberwithin{equation}{section}

\renewcommand{\(}{\left(}
\renewcommand{\)}{\right)}

\def \div{\mbox{div}}

\def \eb{{\bf e}}
\def \nb{{\bf n}}
\def \f{{\bf f}}

\def \x{{\bf x}}
\def \u{{\bf u}}
\def \v{{\bf v}}
\def \w{{\bf w}}
\def \z{{\bf z}}

\def \i{^{-1}}

\def \ch{\chi}
\def \d{\delta}
\def \eps{\epsilon}
\def \D{\Delta}

\def \l{\lambda}
\def \g{\gamma}
\def \th{\theta}

\def \G{\Gamma}
\def \o{\omega}
\def \O{\Omega}

\def \m{\mu}
\def \n{\nu}

\def \ps{\psi}

\def \ph{\phi}

\def \pd{\partial}
\def \<{\langle}
\def \>{\rangle}
\def \nab{\nabla}

\def \nbr{\nonumber}

\def \neq{\not=}

\def \({\Big(}
\def \){\Big)}
\def \[{\Big[}
\def \]{\Big]}
\def \<{\langle}
\def \>{\rangle}

\def \B{{\mathcal B}}
\def \C{{\mathcal C}}

\def \E{{\mathcal E}}

\def \S{{\mathcal S}}

\def \div{\mathrm{div}}
\def \curl{\mathrm{curl}}
\def \and{\mbox{ and }}

\def \ol{\overline}
\def \Rbb{\mathbb{R}}

\def \H{{\bf H}}

\def\bPh{{\boldsymbol \Ph}}

\def\bPs{{\boldsymbol \Ps}}

\def\bT{{\boldsymbol \T}}

\begin{document}
\title[The incompressible limit of the corner singularities]{Convergence in the incompressible limit \\ of the corner singularities}

\author[Choi]{Hyung Jun Choi}
\address{School of Liberal Arts, Korea University of Technology and Education, Cheonan 31253, Republic of Korea}
\email{hjchoi@koreatech.ac.kr; choihjs@gmail.com}

\author[Kim]{Seonghak Kim}
\address{Department of Mathematics, College of Natural Sciences, Kyungpook National University, Daegu 41566, Republic of Korea}
\email{shkim17@knu.ac.kr}

\author[Koh]{Youngwoo Koh}
\address{Department of Mathematics Education, Kongju National University, Kongju 32588, Republic of Korea}
\email{ywkoh@kongju.ac.kr}


\maketitle

\begin{abstract}
In this paper, we treat the corner singularity expansion and its convergence result regarding the penalized system obtained by eliminating the pressure variable in the Stokes problem of incompressible flow.
The penalized problem is a kind of the Lam\'{e} system, so we first discuss the corner singularity theory of the Lam\'{e} system with inhomogeneous Dirichlet boundary condition on a non-convex polygon.
Considering the inhomogeneous condition, we show the decomposition of its solution, composed of singular parts and a smoother remainder near a re-entrant corner, and furthermore, we provide the explicit formulae of coefficients in singular parts. In particular, these formulae can be used in the development of highly accurate numerical scheme.
In addition, we formulate coefficients in singular parts regarding the Stokes equations with inhomogeneous boundary condition and non-divergence-free property of velocity field, and thus we show the convergence results of coefficients in singular parts and remainder regarding the concerned penalized problem.
\end{abstract}


\def \bPh{{\bold\Phi}}
\def \bPs{{\bold\Psi}}
\def \bT{{\bold T}}

\vspace{0.7cm}

\section{Introduction}\label{sec1}

\vspace{0.2cm}

In this paper, we consider the penalized system to the Stokes problem of incompressible flow with inhomogeneous Dirichlet boundary condition (cf. \cite{GRF}). Let $\Omega\subset\mathbb{R}^2$ be a non-convex polygon. Given a small penalty parameter $\eps>0$,
\begin{equation}\label{penalty_stokes}
\left\{\begin{aligned}
-\mu\Delta{\bf u}^\eps+\nabla p^\eps&={\bf f}&&\mbox{ in }\Omega,\\
\div\,\u^\eps&=-\eps p^\eps&&\mbox{ in }\Omega,\\
\u^\eps&={\bf g}&&\mbox{ on }\Gamma:=\partial\O,
\end{aligned}\right.
\end{equation}
where $\u^\eps$ is the approximated velocity field and $p^\eps$ is the approximated pressure; $\mu>0$ is the viscous number; $\f$ and ${\bf g}$ are given external functions.
By eliminating $p^\eps$ in \eqref{penalty_stokes}, we get the following equivalent second order elliptic problem for $\u^\eps$:
\begin{equation}\label{Lame_eps}
\left\{\begin{aligned}
-\mu\Delta\u^\eps-\frac{1}{\eps}\nabla\div\,\u^\eps&=\f&&\mbox{ in }\O,\\
\u^\eps&={\bf g}&&\mbox{ on }\Gamma.
\end{aligned}\right.
\end{equation}
The elimination of the pressure variable in \eqref{Lame_eps} reduces the number of degrees of freedom in the discretization, and thus this reduction provides a more efficient algorithm in the case that the calculation of velocity field is only required. Various studies of its related numerical schemes have been actively investigated (see \cite{ABDA,BSL,CHHF,FN,LA,LLSA,SA}).
In fact, the solution $\u^\eps$ to \eqref{Lame_eps} has a nearly incompressible property for all sufficiently small $\eps>0$, and in particular, its discrete solution converges slowly in the numerical scheme based on standard conforming finite elements. Such phenomena is called {\it numerical locking}; the detailed explanation of locking effects can be found in \cite{AD,BSO,BF,BSL}.

\vspace{0.1cm}

The reduced problem \eqref{Lame_eps} is a kind of the Lam\'{e} system in the isotropic elasticity model (see \cite{GL,GS,LLSA,NA}).
So, this paper concerns a corner singularity expansion of solution to the Lam\'{e} system. As shown in \cite{GL,GS,NA}, such concern and its related regularity issues were previously studied in the use of singular functions of the Lam\'{e} operator. In particular, P. Grisvard analyzed the singular behavior of solution to the Lam\'{e} system near the corners of a plane polygon in \cite{GS} and here, a description of the behavior of solution was also investigated near the edges of a polyhedron.
Historically, the previous references \cite{GL,GS,NA} mainly focused on general theories regarding a corner singularity expansion of the solution for fixed Lam\'{e} constants, but the behavior of its corner singularity solution to \eqref{Lame_eps} for very small $\eps$ was not considered.
On the other hand, it is crucial to investigate the case of very small $\eps$ in order to develop a numerical method overcoming the numerical locking in the penalized problem \eqref{Lame_eps}.
For this reason, in this paper, we make an effort to derive the convergence results of the singular parts and smoother remainder in the corner singularity expansion as the penalty parameter $\eps$ goes to zero.
To do this, we try to construct the explicit formulae of coefficients in the singular part. Actually, these formulae are crucial in studying the convergence of singular parts and that of remainder separately, and it is also confirmed that the formulae are well-defined by showing their a priori estimates.

\vspace{0.1cm}

\begin{figure}[htbp]
\begin{center}
\includegraphics[height=6cm]{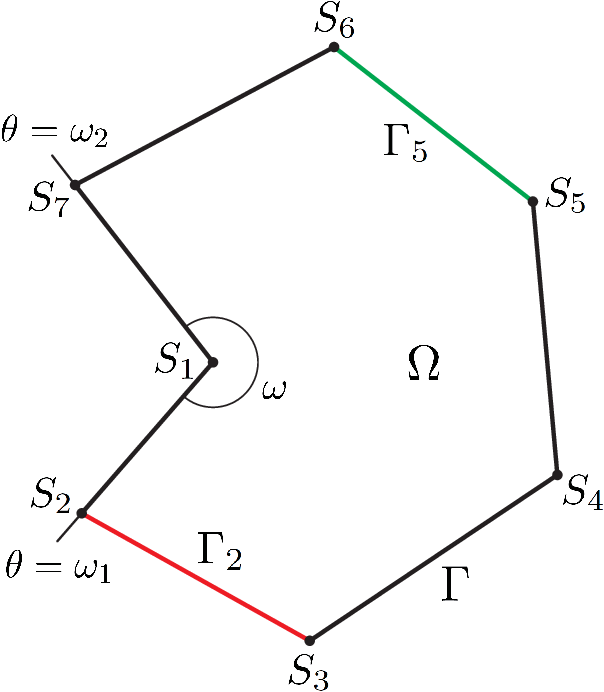}
\end{center}
\caption{The polygonal domain $\Omega$ with a re-entrant corner}\label{fig1}
\end{figure}
For simplicity, we assume in this paper that the boundary $\Gamma$ has only one non-convex vertex $S_1$ placed at the origin (see Figure \ref{fig1}). Denote by $S_j$ $(1\leq j\leq J)$ the vertices of $\Gamma$ with the convention that $S_{J+1}=S_1$. Let $\Gamma_j$ be the side of $\Gamma$ between $S_j$ and $S_{j+1}$, and ${\bf n}_j$ be the outward unit normal vector on $\G_j$. Set ${\bf g}_j:={\bf g}|_{\Gamma_j}$. Let $\omega$ be the re-entrant angle at $S_1$, defined by $\o:=\o_2-\o_1$, where $\o_i$'s are the numbers satisfying $\o_1+\pi<\o_2<\o_1+2\pi$.

\vspace{0.1cm}

First of all, we state the corner singularity expansion of solution to the penalized Lam\'{e} system \eqref{Lame_eps} and describe the explicit formulae of coefficients in its singular part, which will be proved in Section \ref{sec3}.

\begin{thm}\label{thm1.0}
If ${\bf f}\in{\bf L}^2(\Omega)$ and ${\bf g}_j\in{\bf H}^{3/2}(\Gamma_j)$ for $1\leq j\leq J$ satisfy
\begin{equation}\label{g-condition}
\sum_{j=1}^J\int_{\Gamma_j}{\bf g}_j\cdot{\bf n}_j \,dS=0\quad\mbox{ and }\quad{\bf g}_j(S_{j+1})={\bf g}_{j+1}(S_{j+1}),
\end{equation}
then there exists a unique solution $\u^\eps\in\H^1(\O)$ to \eqref{Lame_eps}.
Furthermore, the unique solution $\u^\eps$ can be decomposed as follows (see \cite{GS}): there exist a smoother vector field $\w^\eps\in\H^2(\O)$ and numbers $c^\eps_1$, $c^\eps_2\in\mathbb{R}$ such that
\begin{equation}\label{decomposition_eps}
\u^\eps=\w^\eps+\sum_{i=1}^{2}c^\eps_i\,\bPh_i^\eps,
\end{equation}
where $\bPh_i^\eps\not\in\H^2(\O)$ denotes a singular function for the Lam\'{e} differential operator (to be explicitly defined in Section \ref{sec3}).
In particular, the numbers $c^\eps_1$ and $c^\eps_2$ have the following formulae:
\begin{equation}\label{coeff_eps_formula}
c^\eps_1=\frac{\C_1^\eps(\f,{\bf g})}{\g^\eps_1}\quad\mbox{and}\quad c_2^\eps=\frac{\C_2^\eps(\f,{\bf g})+c_1^\eps\C_*^\eps}{\g_2^\eps},
\end{equation}
where
\begin{equation}\label{gamma_eps}
\begin{aligned}
\g_i^\eps =& \,\m\int_{\o_1}^{\o_2}\left(2\l_i^\eps\,\bT_i^\eps(\th)\cdot{\widetilde\bT}_i^{\eps}(\th)\right)d\th+\eps^{-1}\int_{\o_1}^{\o_2}\left(I_{i}^\eps(\th)+J_{i}^\eps(\th)\right)d\th,\\
\mathcal{C}_i^\eps(\f,{\bf g}) =&\, \int_\O\f\cdot\left({\widetilde\bPh}_i^\eps+\bPs_i^\eps\right)d\x\\
&\, -\sum_{j=1}^J\int_{\G_j}\left(\m\,{\bf g}_j\cdot\frac{\partial}{\partial{\bf n}_j}\left({\widetilde\bPh}_i^\eps+\bPs_i^\eps\right)+\eps^{-1}\left({\bf g}_j\cdot{\bf n}_j\right)\div\left({\widetilde\bPh}_i^\eps+\bPs_i^\eps\right)\right)dS,\\
\mathcal{C}_*^\eps= &\, \sum_{j=2}^{J-1}\int_{\G_j}\left(\m\,\bPh_1^\eps\cdot\frac{\partial}{\partial\nb_j}\left({\widetilde\bPh}_2^\eps+\bPs_2^\eps\right)+\eps^{-1}\left(\bPh_1^\eps\cdot\nb_j\right)\div\left({\widetilde\bPh}_2^\eps+\bPs_2^\eps\right)\right)dS,\\
I_{i}^\eps(\th)=&\,2\l_i^\eps[\bT_i^\eps(\th)\cdot\eb_r(\th)][{\widetilde\bT}_i^\eps(\th)\cdot\eb_r(\th)],\\
J_{i}^\eps(\th)=&\,[\pd_\th\bT_i^\eps(\th)\cdot\eb_\th(\th)][{\widetilde\bT}_i^\eps(\th)\cdot\eb_r(\th)]-[\bT_i^\eps(\th)\cdot\eb_r(\th)][\pd_\th{\widetilde\bT}_i^\eps(\th)\cdot\eb_\th(\th)].
\end{aligned}
\end{equation}
Here, $\l_i^\eps\in(1/2,1)$ is an eigenvalue regarding the Lam\'{e} operator, and $\bT_i^\eps(\th)$ is a corresponding eigenfunction.
Also, ${\widetilde\bPh}_i^\eps$ and ${\widetilde\bT}_i^\eps(\th)$ are obtained by replacing $\l_i^\eps$ with $-\l_i^\eps$ in the definition of $\bPh_i^\eps$ and $\bT_i^\eps(\th)$, respectively.
In addition, $\bPs_i^\eps\in{\bf H}^1(\Omega)$ is the weak solution to
\begin{equation}\label{Psi-eps-prob}
\left\{\begin{aligned}
 -\mu\Delta\bPs_i^\eps-\eps^{-1}\nabla\div\,\bPs_i^\eps&={\bf 0}&&\mbox{ in }\Omega,\\
 \bPs_i^\eps&={\bf 0}&&\mbox{ on }\Gamma_j,\quad j=1,J,\\
 \bPs_i^\eps&=-{\widetilde\bPh}_i^\eps&&\mbox{ on }\Gamma_j,\quad j=2,\ldots,J-1.
\end{aligned}\right.
\end{equation}
\end{thm}

\vspace{0.1cm}

\begin{rem}\label{rk1.2}
Considering the expressions of $\gamma_i^\eps$, $\C_i^\eps(\f,{\bf g})$ and $\C_*^\eps$ in \eqref{gamma_eps}, it is expected that they may blow up as $\eps\rightarrow 0^+$, due to the presence of the term $\eps^{-1}$ in their formulae.
To avoid such blow-up phenomena, we try not to use a cutoff function for $\bPh_i^\eps$ and ${\widetilde\bPh}_i^\eps$ in this paper (cf. \cite{CKT}).
Without using a cutoff function, the following calculation holds:
\begin{equation}\label{remark-eq1}
\begin{aligned}
I_i^\eps(\th)+J_i^\eps(\th)&=\eps\times\mathcal{K}_i^\eps(\th),\\
\div\,{\widetilde\bPh}_i^\eps&=\eps\times r^{-\l_i^\eps-1}{\widetilde{\mathcal{K}}}_i^\eps(\th),
\end{aligned}
\end{equation}
where $\mathcal{K}_i^\eps(\th)$ and ${\widetilde{\mathcal{K}}}_i^\eps(\th)$ denote uniformly bounded functions for all $\eps>0$, explicitly described in \eqref{gamma1-2}, \eqref{gamma2-2} and \eqref{eps_div_dual}.
It is also observed that
\begin{equation}\label{remark-eq2}
\|\div\,\bPs_i^\eps\|_{L^2(\Gamma_j)}\leq C\eps\quad\mbox{for some constant $C>0$},
\end{equation}
which is mentioned in \eqref{G_plus_ij}.
By \eqref{remark-eq1} and \eqref{remark-eq2}, the quantities $\gamma_i^\eps$, $\C_i^\eps(\f,{\bf g})$ and $\C_*^\eps$ do not blow up even though $\eps$ is nearly zero. Thus, the formulae of $c_1^\eps$ and $c_2^\eps$ in \eqref{coeff_eps_formula} are well-defined.
\end{rem}

\vspace{0.1cm}

Next, we intend to generalize some corner singularity results of the Stokes equations in \cite{CKT} by  considering further the inhomogeneous boundary condition.
This generalization is necessary in regard to the convergence of $\w^\eps$ and $c_i^\eps$ in \eqref{decomposition_eps}.
Concerning the penalized problem \eqref{penalty_stokes}, we note that its solution $[\u^\eps,\,p^\eps]$ converges to the solution satisfying the Stokes problem of incompressible flow in ${\bf H}^1(\O)\times{\rm L}^2(\O)$ as $\eps\rightarrow 0^+$ (see Theorem 5.3 in \cite[Chapter I]{GRF}); so it is also expected that the singular parts and remainder of $[\u^\eps,\,p^\eps]$ satisfying \eqref{penalty_stokes} converge to those of the Stokes equations, respectively.
Now, we consider the following Stokes equations with non-divergence-free property and inhomogeneous boundary condition:
\begin{equation}\label{Stokes}
\left\{\begin{aligned}
-\mu\Delta{\bf u}^{{\rm s}}+\nabla p^{{\rm s}}&={\bf f}&&\mbox{ in }\Omega,\\
\div\,\u^{{\rm s}}&=\zeta&&\mbox{ in }\Omega,\\
\u^{{\rm s}}&={\bf g}&&\mbox{ on }\Gamma,
\end{aligned}\right.
\end{equation}
where $\u^{{\rm s}}$ is the velocity field and $p^{{\rm s}}$ is the pressure; $\mu>0$ is the viscous number; $\f$, $\zeta$ and ${\bf g}$ are given functions with $\int_\O \zeta\,d\x=\int_\Gamma\left({\bf g}\cdot{\bf n}\right)dS$.

\vspace{0.1cm}

As similar to Theorem \ref{thm1.0}, we state the corner singularity expansion of solution to the Stokes problem \eqref{Stokes} and also derive the formulae of coefficients in singular parts, which will be shown in Section \ref{sec4}.

\begin{thm}\label{thm1.3}
If $\f\in{\bf L}^2(\O),$ $\zeta\in{\rm H}^{1}(\O)$ with $\zeta(S_1)=0$ and ${\bf g}_j\in{\bf H}^{3/2}(\Gamma_j)$ for $1\leq j\leq J$ with \eqref{g-condition} and ${\bf g}_1(S_1)={\bf g}_J(S_1)={\bf 0}$, then there exists a unique solution pair $[\u^{{\rm s}},\,p^{{\rm s}}]\in{\bf H}^1(\O)\times{\rm L}_0^2(\O)$ of the Stokes problem \eqref{Stokes} that can be written as follows: for some numbers $c^{\rm s}_1, c^{\rm s}_2\in\mathbb{R}$ and pair $\left[\w^{\rm s},\,\sigma^{\rm s}\right]\in{\bf H}^2(\O)\times{\rm H}^{1}(\O),$
\begin{equation}\label{Stokes-decomposition-sec1}
\left[\u^{{\rm s}},\,p^{{\rm s}}\right]=\left[\w^{\rm s},\,\sigma^{\rm s}\right]+\sum_{i=1}^{2}c^{\rm s}_i\left[\bPh^{\rm s}_i,\,\phi^{\rm s}_i\right],
\end{equation}
where $\left[\bPh_i^{\rm s},\,\phi_i^{\rm s}\right]\not\in\H^2(\O)\times{\rm H}^1(\O)$ denotes a singular function pair of the Stokes differential operator (to be defined explicitly in Section \ref{sec4}).
Furthermore, the coefficients $c^{\rm s}_1$ and $c^{\rm s}_2$ in \eqref{Stokes-decomposition-sec1} have the following formulae:
\begin{equation}\label{Stokes-sif-formula-sec1}
 c^{\rm s}_1=\frac{\mathcal{C}^{\rm s}_1(\f,\zeta,{\bf g})}{\g^{\rm s}_1}\quad\mbox{and}\quad
 c^{\rm s}_2=\frac{\mathcal{C}^{\rm s}_2(\f,\zeta,{\bf g})+c^{\rm s}_1\mathcal{C}^{\rm s}_*}{\g^{\rm s}_2},
\end{equation}
where
\begin{equation}\label{Stokes-C}
\begin{aligned}
\g^{\rm s}_i:=&\,\int_{\o_1}^{\o_2}\left(2\kappa_i\bT^{\rm s}_i(\th)\cdot{\widetilde\bT}^{\rm s}_i(\th)-\xi_i^{\rm s}(\th)\,[{\widetilde\bT}^{\rm s}_i(\th)\cdot\eb_r(\th)]+[\bT_i^{\rm s}(\th)\cdot\eb_r(\th)]\,{\widetilde\xi}^{\rm s}_i(\th)\right)d\th,\\
\mathcal{C}^{\rm s}_i(\f,\zeta,{\bf g}):=&\,\int_\O\left(\f\cdot\left(\mu{\widetilde\bPh}^{\rm s}_i+\bPs^{\rm s}_i\right)-\zeta\left(\mu{\widetilde\ph}^{\rm s}_i+\ps^{\rm s}_i\right)\right)d\x\\
 &\,-\sum_{j=1}^J\int_{\G_j}\left(\m\,{\bf g}_j\cdot\frac{\partial}{\partial{\bf n}_j}\left(\mu{\widetilde\bPh}^{\rm s}_i+\bPs^{\rm s}_i\right)-\left({\bf g}_j\cdot{\bf n}_j\right)\left(\mu{\widetilde\ph}^{\rm s}_i+\ps^{\rm s}_i\right)\right)dS,\\
 \mathcal{C}^{\rm s}_*:=&\,\sum_{j=2}^{J-1}\int_{\G_j}\left(\m\bPh^{\rm s}_1\cdot\frac{\partial}{\partial\nb_j}\left(\mu{\widetilde\bPh}^{\rm s}_2+\bPs^{\rm s}_2\right)-\left(\bPh^{\rm s}_1\cdot\nb_j\right)\left(\mu{\widetilde\ph}^{\rm s}_2+\ps^{\rm s}_2\right)\right)dS,
\end{aligned}
\end{equation}
and $\left[\bPs^{\rm s}_i,\,\psi^{\rm s}_i\right]\in{\bf H}^1(\Omega)\times{\rm L}_0^2(\Omega)$ is the weak solution pair of
\begin{equation}\label{Stokes-ps}
\left\{\begin{aligned}
 -\mu\Delta\bPs^{\rm s}_i+\nabla\psi^{\rm s}_i&={\bf 0}&&\mbox{ in }\Omega,\\
 \div\,\bPs^{\rm s}_i&=0&&\mbox{ in }\Omega,\\
 \bPs^{\rm s}_i&={\bf 0}&&\mbox{ on }\Gamma_j,\quad j=1,J,\\
 \bPs^{\rm s}_i&=-\mu{\widetilde\bPh}^{\rm s}_i&&\mbox{ on }\Gamma_j,\quad j=2,\cdots,J-1.
\end{aligned}\right.
\end{equation}
Here, $\kappa_i\in(1/2,1)$ denotes an eigenvalue of the Stokes differential operator, and $\left[\bT_i^{\rm s}(\th),\,\xi_i^{\rm s}(\th)\right]$ is a corresponding eigenfunction pair.
Also, $[{\widetilde\bPh}_i^{\rm s},\,{\widetilde\phi}_i^{\rm s}]$ and $[{\widetilde\bT}_i^{\rm s}(\th),\,{\widetilde\xi}_i^{\rm s}(\th)]$ are obtained by replacing $\kappa_i$ with $-\kappa_i$ in the definition of $[\bPh_i^{\rm s},\,\phi_i^{\rm s}]$ and $[\bT_i^{\rm s}(\th),\,\xi_i^{\rm s}(\th)]$, respectively.
\end{thm}

\vspace{0.1cm}

\begin{rem}\label{rk1.4}
Having compared with Theorem \ref{thm1.0} for the penalized Lam\'{e} system, it makes sense that the two functions $\bPh_i^{\rm s}$ and ${\widetilde\bPh}_i^{\rm s}$ in Theorem \ref{thm1.3} do not contain a cutoff function.
Due to avoiding cutoff function, we remark that these constructions of singular functions are different from the ones used in \cite{CKT}.
Such constructions to avoid cutoff function are used in the previous papers \cite{GL,GS,NA}, and a similar idea can be also found in the numerical scheme: {\it singular complement method} (see \cite{CHT,CJKLZ1,CJKLZ2}).
\end{rem}

\vspace{0.1cm}

Finally, we present the following main result: the convergence of $\w^\eps$ and $c_i^\eps$ in \eqref{decomposition_eps} as $\eps\rightarrow 0^+$ to be verified in Section \ref{sec5}.

\begin{thm}\label{thm1.1}
Given $\eps>0$, let $c_1^\eps$, $c_2^\eps$ be the real-valued coefficients of the singular part in the corner singularity expansion \eqref{decomposition_eps} with the formulae defined by \eqref{coeff_eps_formula}.
Also, let $\w^\eps\in{\bf H}^2(\O)$ be the smoother remainder in \eqref{decomposition_eps} and $\sigma^\eps:=-(1/\eps)\,\div\,\w^\eps\in{\rm H}^1(\O)$. Then, for $i=1,2$,
\begin{align}
(i)&~c_i^\eps\rightarrow \mu^{-1} c_i^{\rm s}&\mbox{ as }\,\eps\rightarrow 0^+,\label{sif_eps_convergence}\\
(ii)&~\mu\|\w^\eps-\w^{\rm s}\|_{{\bf H}^1(\O)}+\|\sigma^\eps-\sigma^{\rm s}\|_{{\rm L}^2(\O)}\rightarrow 0&\mbox{ as }\,\eps\rightarrow 0^+,\label{regular_eps_convergence}
\end{align}
where the numbers $c_1^{\rm s},c_2^{\rm s}\in\mathbb{R}$, formulated by \eqref{Stokes-sif-formula-sec1}, are the coefficients of the singular part to the Stokes problem of incompressible flow, and the pair $[\w^{\rm s},\sigma^{\rm s}]\in{\bf H}^2(\O)\times{\rm H}^1(\O)$, given in \eqref{Stokes-decomposition-sec1}, denotes the regular part to the Stokes problem.
\end{thm}

\vspace{0.1cm}

Throughout this paper, we will use the following standard notations (cf. \cite{AS,GRF,GE}).
For $s\geq 0$ and $p\geq 1$, let ${\rm W}^{s,p}(\mathcal{O})$ be a fractional Sobolev space on an open set $\mathcal{O}$ with the norm $\|\cdot\|_{s,p,\mathcal{O}}$.
In the case of $p=2$, we use a simple notation ${\rm H}^s(\mathcal{O})={\rm W}^{s,2}(\mathcal{O})$ with the norm $\|\cdot\|_{s,\mathcal{O}}$.
Denote by $C_0^\infty(\Omega)$ the linear space of infinitely differentiable functions with compact support in $\Omega$. We define
$$
{\rm H}_0^s(\Omega)=\overline{C_0^\infty(\Omega)}^{{\rm H}^s(\Omega)},
$$
i.e., ${\rm H}_0^s(\Omega)$ is the closure of $C_0^\infty(\Omega)$ in the space ${\rm H}^s(\Omega)$.
Let ${\rm H}^{-s}(\Omega)$ be the dual space of ${\rm H}_0^s(\Omega)$ normed by
$$
\|f\|_{-s,\Omega}=\sup_{0\neq v\in {\rm H}_0^s(\Omega)}\frac{\langle f,v\rangle}{\|v\|_{s,\Omega}},
$$
where $\langle~,~\rangle$ means the duality pairing between ${\rm H}^{-s}(\Omega)$ and ${\rm H}^s_0(\Omega)$.
Let ${\rm L}_0^2(\Omega)=\{v\in {\rm L}^2(\Omega):\int_\Omega v \,d{\bf x}=0\}$ and $\bar{{\rm H}}^s(\Omega)={\rm H}^s(\Omega)\cap {\rm L}_0^2(\Omega)$.
Furthermore, denote by $\H^s(\O)={\rm H}^s(\O)\times{\rm H}^s(\O)$ and similar for other spaces.
We shall also use a generic constant $C>0$ which is only depending on $\Omega$.

\vspace{0.1cm}

The rest of the paper is organized as follows.
In Section \ref{sec2}, we give a basic regularity result for the solution to the Lam\'{e} system with the inhomogeneous Dirichlet boundary condition. Section \ref{sec3} formulates explicitly coefficients in singular parts to the Lam\'{e} system and demonstrates their well-posedness. Section \ref{sec4} is devoted to explicit formulation and well-posedness of coefficients in singular parts to the Stokes problem with non-divergence-free property and inhomogeneous boundary condition. Finally, the main result of the paper, Theorem \ref{thm1.1}, is proved in  Section \ref{sec5}.

\vspace{0.7cm}

\section{Preliminaries}\label{sec2}

\vspace{0.2cm}

In this section, we discuss a basic regularity result on the Lam\'{e} system regarding the isotropic elasticity model. The concerned problem with inhomogeneous Dirichlet boundary condition is
\begin{equation}\label{Lame}
\left\{\begin{aligned}
-\mu\Delta\u-(\mu+\n)\nabla\div\,\u&=\f&&\mbox{ in }\Omega,\\
\u&={\bf g}&&\mbox{ on }\Gamma,
\end{aligned}\right.
\end{equation}
where $\u$ is the displacement, $\f$ is the body force in the elastic material occupying the bounded polygonal region $\O\subset\mathbb{R}^2$, and $\mu,\n>0$ are the Lam\'{e} constants with $\mu_1\leq\mu\leq\mu_2$ for some constants $\mu_2\ge\mu_1>0$.

\vspace{0.1cm}

First, we give a property of the divergence operator in the following lemma.

\begin{lem}\label{lem2.1}
Assume that $\v\in{\bf H}^s(\Omega)\cap{\bf H}_0^1(\Omega)$ for some $1\leq s\leq 2$.
Then, there exists $\w\in{\bf H}^s(\Omega)\cap{\bf H}_0^1(\Omega)$ such that
$$
\div \,\w=\div\, \v\quad\mbox{in $\Omega$}
$$
and
$$
\|\w\|_{s,\Omega}\leq C\|\div\,\v\|_{s-1,\Omega},
$$
where $C=C(\Omega)$ is a positive constant.
\end{lem}

\begin{proof}
The case of integer values $s=1,2$ is obviously covered by \cite[Lemma 2.1]{BSL}.
Furthermore, the result for any $s\in(1,2)$ can be obtained by the interpolation theory between the integer values $s=1,2$.
\end{proof}

\vspace{0.1cm}

Next, we give a priori estimate of the solution $\u$ to the problem \eqref{Lame} (cf. \cite{BSL,GRF,GS}).

\begin{thm}\label{thm2.2}
If $\f\in{\bf H}^{-1}(\Omega)$ and ${\bf g}_j\in{\bf H}^{1/2}(\Gamma_j)$ for $1\leq j\leq J$, then there exists a unique solution $\u\in{\bf H}^1(\Omega)$ to the boundary value problem \eqref{Lame}.
Furthermore, let $1<s<1+\min\{\lambda_1,\kappa_1\}$, where $\lambda_1,\kappa_1\in(1/2,1)$ are the leading singular exponents of the Lam\'{e} operator and  of the Stokes operator, defined in Sections \ref{sec3} and \ref{sec4}, respectively.
If $\f\in{\bf H}^{s-2}(\Omega)$ and ${\bf g}_j\in{\bf H}^{s-1/2}(\Gamma_j)$ for $1\leq j\leq J$ satisfy \eqref{g-condition}, then the solution $\u$ to \eqref{Lame} belongs to ${\bf H}^s(\Omega)$ and fulfills the following a priori estimate:
\begin{equation}\label{hs-estimate}
\mu\|\u\|_{s,\Omega}+\left(\mu+\n\right)\|\div\,\u\|_{s-1,\Omega}\leq C\bigg(\|\f\|_{s-2,\Omega}+\m\sum_{j=1}^J\|{\bf g}_j\|_{s-1/2,\Gamma_j}\bigg),
\end{equation}
where $C=C(\O)$ is a positive constant independent of the Lam\'{e} constants $\m$ and $\nu$.
\end{thm}

\begin{proof}
By the Lax-Milgram theorem in \cite{EP}, the existence and uniqueness of a solution $\u\in{\bf H}^1(\Omega)$ to the problem \eqref{Lame} are obviously guaranteed (cf. \cite{GRF}).
So it remains to show the estimate \eqref{hs-estimate}.

Consider the following homogeneous Dirichlet boundary value problem:
\begin{equation}\label{Lame-homogeneous}
\left\{\begin{aligned}
-\mu\Delta\u_0-(\mu+\n)\nabla\div\,\u_0&=\f_0&&\mbox{ in }\Omega,\\
\u_0&={\bf 0}&&\mbox{ on }\Gamma.
\end{aligned}\right.
\end{equation}
Let $1<s<1+\min\{\lambda_1,\kappa_1\}$. If $\f_0\in{\bf H}^{s-2}(\Omega)$, then, by \cite{GS}, the solution $\u_0$ to \eqref{Lame-homogeneous} belongs to ${\bf H}^s(\Omega)\cap{\bf H}_0^1(\Omega)$, and thus Lemma \ref{lem2.1} guarantees the existence of a function ${\boldsymbol\zeta}\in{\bf H}^s(\Omega)\cap{\bf H}_0^1(\Omega)$ such that
$$
\div\,{\boldsymbol\zeta}=\div\,\u_0\quad\mbox{ in }\Omega,
$$
and furthermore,
\begin{equation}\label{zeta-estimate}
\|{\boldsymbol\zeta}\|_{s,\Omega}\leq C\|\div\,\u_0\|_{s-1,\Omega}.
\end{equation}
Letting
\begin{equation}\label{u1-p}
\u_1:=\u_0-{\boldsymbol\zeta}\quad\mbox{ and }\quad p_1:=-\left(\frac{\mu+\n}{\mu}\right)\div\,\u_0,
\end{equation}
the problem \eqref{Lame-homogeneous} is equivalent to the following Stokes problem:
$$
\left\{\begin{aligned}
-\D\u_1+\nab p_1&=\mu^{-1}\f_0+\D{\boldsymbol\zeta}&&\mbox{ in }\Omega,\\
\div\,\u_1&=0&&\mbox{ in }\Omega,\\
\u_1&={\bf 0}&&\mbox{ on }\Gamma.
\end{aligned}\right.
$$
From \cite[Theorem 2.1]{CKT},
\begin{equation}\label{u1-p-estimate}
\|\u_1\|_{s,\Omega}+\|p_1\|_{s-1,\Omega}\leq C\mu^{-1}\|\f_0\|_{s-2,\Omega}+C\|\D{\boldsymbol\zeta}\|_{s-2,\Omega}.
\end{equation}
By \eqref{u1-p-estimate} and \eqref{zeta-estimate},
\begin{align}
\m\|\u_0\|_{s,\Omega}+(\m+\n)\|\div\,\u_0\|_{s-1,\Omega}
&\leq\m\left(\|\u_1\|_{s,\Omega}+\|{\boldsymbol\zeta}\|_{s,\Omega}+\|p_1\|_{s-1,\Omega}\right)\nonumber\\
&\leq C\left(\|\f_0\|_{s-2,\Omega}+\m\|{\boldsymbol\zeta}\|_{s,\Omega}\right)\nonumber\\
&\leq C_1\left(\|\f_0\|_{s-2,\Omega}+\m\|\div\,\u_0\|_{s-1,\Omega}\right),\label{u-divu-estimate}
\end{align}
where $C_1>0$ is a constant independent of $\m$ and $\n$.
If $\n\geq\m(2C_1-1)$, then the inequality \eqref{u-divu-estimate} becomes
\begin{equation}\label{u-divu-s-estimate}
\m\|\u_0\|_{s,\Omega}+\left(\frac{\m+\n}{2}\right)\|\div\,\u_0\|_{s-1,\Omega}\leq C_1\|\f_0\|_{s-2,\O}.
\end{equation}
On the other hand, if $\n<\m(2C_1-1)$, then the estimate \eqref{u-divu-s-estimate} also follows from the regularity result: $\|\u_0\|_{s,\O}\leq C\m\i\|\f_0\|_{s-2,\O}$ given by \cite[Lemma 2.1]{KF}.
Hence, one can show the desired estimate \eqref{hs-estimate} for any $\mu$ and $\nu$  when ${\bf g}_j\equiv{\bf 0}$ for $1\leq j\leq J$.

\vspace{0.1cm}

Next, we consider the inhomogeneous Dirichlet boundary value problem \eqref{Lame}.
Since ${\bf g}_j\in{\bf H}^{s-1/2}(\Gamma_j)$ for $1\leq j\leq J$ satisfy \eqref{g-condition}, by \cite[Theorem 6.2]{ASVR}, there exists a function ${\bf z}\in{\bf H}^{s+1}(\Omega)$ satisfying $\curl\,{\bf z}|_{\Gamma_j}={\bf g}_j$ for $1\leq j\leq J$ and
\begin{equation}\label{z-estimate}
\|{\bf z}\|_{s+1,\Omega}\leq C\sum_{j=1}^J\|{\bf g}_j\|_{s-1/2,\Gamma_j}.
\end{equation}
Letting $\u_2:=\u-\curl\,{\bf z}$, the concerned problem \eqref{Lame} is rewritten as
\begin{equation}\label{Lame-u2}
\left\{\begin{aligned}
-\mu\Delta\u_2-(\mu+\n)\nabla\div\,\u_2&=\f+\m\D\curl\,{\bf z}&&\mbox{ in }\Omega,\\
\u_2&={\bf 0}&&\mbox{ on }\Gamma.
\end{aligned}\right.
\end{equation}
From \eqref{u-divu-s-estimate}, the solution $\u_2$ to \eqref{Lame-u2} satisfies
\begin{equation}\label{u2-estimate}
\m\|\u_2\|_{s,\O}+\left(\m+\n\right)\|\div\,\u_2\|_{s-1,\O}\leq C_1\left(\|\f\|_{s-2,\O}+\m\|{\bf z}\|_{s+1,\O}\right).
\end{equation}
Since $\u=\u_2+\curl\,{\bf z}$ in $\O$, by \eqref{u2-estimate} and \eqref{z-estimate}, we obtain
$$
\begin{aligned}
\m\|\u\|_{s,\O}+\left(\m+\n\right)\|\div\,\u\|_{s-1,\O}
&\leq\m\left(\|\u_2\|_{s,\O}+\|{\bf z}\|_{s+1,\O}\right)+\left(\m+\n\right)\|\div\,\u_2\|_{s-1,\O}\\
&\leq C\left(\|\f\|_{s-2,\O}+\m\|{\bf z}\|_{s+1,\O}\right)\\
&\leq C\bigg(\|\f\|_{s-2,\O}+\m\sum_{j=1}^J\|{\bf g}_j\|_{s-1/2,\G_j}\bigg),
\end{aligned}
$$
yielding the desired estimate \eqref{hs-estimate}.
\end{proof}

\vspace{0.7cm}

\section{Formulation of coefficients in corner singularities of the Lam\'{e} system}\label{sec3}

\vspace{0.2cm}

In this section, we first discuss the corner singularities of the Lam\'{e} system \eqref{Lame} near the re-entrant corner and then describe the corner singularity expansion of its solution (see \cite{GL,GS,KMRS,NA}).
The corner singularity expansion means that the solution can be decomposed into the two parts: the regular part belonging to ${\bf H}^2(\O)$ and singular part outside ${\bf H}^2(\O)$.
Such corner singularity expansion is indeed useful in the development of highly accurate numerical scheme near the re-entrant corner, and in particular, the derivation of explicit formulae of coefficients in singular part can improve the convergence rate of discrete solution (cf. \cite{BDLNC,BM,CKA}).
For this reason, we try to provide the well-defined explicit formulae of coefficients in singular part.

\vspace{0.1cm}

We now introduce singular functions for the Lam\'{e} operator $\mathcal{L}:=-\mu\Delta-(\mu+\nu)\nab\div$ with the Dirichlet boundary condition.
Referring to \cite[Section 3.1]{KMRS}, the singular functions are obtained by finding non-trivial solutions of the homogeneous Lam\'{e} system on the sector $\S=\{(r,\th):r>0\mbox{ and }\omega_1<\theta<\omega_2\}$.
The description of singular function is as follows.

Let $\lambda_{i}$ be an eigenvalue of the Lam\'{e} operator, with the Dirichlet boundary condition, satisfying
\begin{equation}\label{eigen-eq}
\left(\frac{3\mu+\nu}{\mu+\nu}\right)^2\sin^2(\lambda_{i}\,\omega)-\lambda_{i}^2\sin^2\omega=0.
\end{equation}
Actually, the first three eigenvalues are real, which are ordered as
\begin{equation}\label{Eigenvalue}
1/2<\lambda_1<\pi/\omega<\lambda_2<1<\lambda_3<2\pi/\omega.
\end{equation}
If the singular part of solution is defined to be outside ${\bf H}^2(\O)$, then we only consider the first two eigenvalues $\l_1$ and $\l_2$ less than one.
Letting $N=2$ and corresponding to each eigenvalue $\l_{i}$ $(1\leq i\leq N)$, the singular function $\bPh_i$ is defined as
\begin{equation}\label{Singular}
\bPh_i(r,\th)=r^{\lambda_{i}}\bT_i(\th),
\end{equation}
where $\bT_i(\th)$ is the trigonometric eigenfunction given by
\begin{equation}\label{T-def}
\bT_i(\th)=A_i(\th-\ol\o)\,\eb_r(\th)+B_i(\th-\ol\o)\,\eb_\th(\th).
\end{equation}
Here, $\ol\o:=\left(\o_1+\o_2\right)/2$, $\eb_r(\th):=(\cos\th,\sin\th)^t$, $\eb_\th(\th):=(-\sin\th,\cos\th)^t$,
$$
\begin{aligned}
\begin{pmatrix}
A_1(\th)\\B_1(\th)
\end{pmatrix}
:=&\,-\begin{pmatrix}
(C_{\mu,\nu}-\lambda_{1})\sin[(1-\lambda_{1})\th]\\
(C_{\mu,\nu}+\lambda_{1})\cos[(1-\lambda_{1})\th]
\end{pmatrix}\\
&\,+(C_{\mu,\nu}\cos(\lambda_{1}\,\omega)+\lambda_{1}\cos\omega)\begin{pmatrix}
\sin[(1+\lambda_{1})\th]\\
\cos[(1+\lambda_{1})\th]
\end{pmatrix},\\
\begin{pmatrix}
A_2(\th) \\ B_2(\th)
\end{pmatrix}:=&\,-\begin{pmatrix}
(C_{\mu,\nu}-\lambda_{2})\cos[(1-\l_{2})\th]\\
-(C_{\mu,\nu}+\lambda_{2})\sin[(1-\l_{2})\th]
\end{pmatrix}\\
&+(C_{\mu,\nu}\cos(\lambda_{2}\,\omega)-\lambda_{2}\cos\omega)\begin{pmatrix}
\cos[(1+\l_{2})\th]\\
-\sin[(1+\l_{2})\th]
\end{pmatrix},
\end{aligned}
$$
and $C_{\mu,\nu}:=\left(3\mu+\nu\right)/\left(\mu+\nu\right)$.
By a direct calculation, it is noted that $\mathcal{L}\bPh_i\equiv 0$ in $\O$ and $\bPh_i\in\H^s\left(\Omega\right)$ for $s<1+\l_{i}$; that is,
$$
\|\bPh_i\|_{s,\Omega}\leq C,
$$
where $C=C(\O)$ is a positive constant independent of the Lam\'{e} constants $\m$ and $\n$.

\vspace{0.1cm}

Next, we describe the corner singularity expansion regarding the Lam\'{e} system \eqref{Lame} in the use of $\bPh_i$ (cf. \cite[Theorem 2.1]{KF}).

\begin{thm}\label{thm3.1}
If $\f\in{\bf L}^2(\Omega)$ and ${\bf g}_j\in{\bf H}^{3/2}(\G_j)$ for $1\leq j\leq J$ with \eqref{g-condition},
%
%
then the solution $\u\in{\bf H}^1(\Omega)$ to the problem \eqref{Lame} has a corner singularity expansion as follows: there exist a regular part $\w\in{\bf H}^2(\Omega)$ and coefficients $c_1,\,c_2\in\Rbb$ such that
\begin{equation}\label{decomposition}
 \u=\w+\sum_{i=1}^N c_i{\bold\Phi}_i.
\end{equation}
\end{thm}

\begin{proof}
As seen in \eqref{z-estimate}, \cite[Theorem 6.2]{ASVR} also gives the existence of a function ${\bf z}\in{\bf H}^3(\Omega)$ satisfying $\curl\,{\bf z}|_{\G_j}={\bf g}_j$ for $1\leq j\leq J$ and
\begin{equation}\label{z-3order-estimate}
\|{\bf z}\|_{3,\Omega}\leq C\sum_{j=1}^J\|{\bf g}_j\|_{3/2,\Gamma_j}.
\end{equation}
With $\u_3:=\u-\curl\,{\bf z}$, the problem \eqref{Lame} can be rewritten as
\begin{equation}\label{Lame-u3}
\left\{\begin{aligned}
\mathcal{L}\u_3&=\f+\m\D\curl\,{\bf z}&&\mbox{ in }\O,\\
\u_3&={\bf 0}&&\mbox{ on }\G,
\end{aligned}\right.
\end{equation}
which is the homogeneous Dirichlet boundary value problem.
Since $\f+\m\D\curl\,{\bf z}\in{\bf L}^2(\Omega)$, by \cite[Theorem 2.1]{KF}, the existence and uniqueness of a solution $\u_3\in{\bf H}_0^1(\Omega)$ to the problem \eqref{Lame-u3} are guaranteed, and furthermore, there exist a smoother function $\u_{\rm R}\in{\bf H}^2(\Omega)\cap{\bf H}_0^1(\Omega)$ and coefficients $c_1,c_2\in\mathbb{R}$ that compose $\u_3$ as
\begin{equation}\label{u3-decomposition}
\u_3=\u_{\rm R}+\sum_{i=1}^N c_i\ch\bPh_i,
\end{equation}
where $\ch=\ch(r)$ is a smooth cutoff function whose value is $1$ near the re-entrant corner $S_1$ and vanishes outside a small neighborhood of $S_1$.
Define
\begin{equation}\label{w-def}
\w:=\u_{\rm R}+\curl\,{\bf z}+\sum_{i=1}^N c_i(\ch-1)\bPh_i.
\end{equation}
Since $\u=\u_3+\curl\,{\bf z}$ in $\O$, by \eqref{u3-decomposition}, the desired expansion \eqref{decomposition} can be shown as follows:
$$
\u=\u_{\rm R}+\curl\,{\bf z}+\sum_{i=1}^N c_i\ch\bPh_i=\w+\sum_{i=1}^N c_i\bPh_i.
$$
Here, the function $(\ch-1)\bPh_i$ in \eqref{w-def} belongs to $\H^2(\O)$ because it vanishes near the re-entrant corner $S_1$.
Combining this property with $\u_{\rm R}+\curl\,{\bf z}\in{\bf H}^2(\Omega)$, one sees that $\w\in{\bf H}^2(\Omega)$, concluding the proof.
\end{proof}

\vspace{0.1cm}

Now, we introduce the explicit formulae of $c_i$ in \eqref{decomposition}.
To derive their formulae, we need to define the dual singular function ${\widetilde\bPh}_i$ as
\begin{equation}\label{Dual}
{\widetilde\bPh}_i(r,\th)=r^{-\lambda_i}{\widetilde\bT}_i(\th),
\end{equation}
where ${\widetilde\bT}_i(\th)$ is obtained by replacing $\lambda_i$ with $-\lambda_i$ in the eigenfunction $\bT_i(\th)$ of \eqref{T-def}.
Using ${\widetilde\bPh}_i$ given in \eqref{Dual}, we can formulate the coefficients $c_i$ in the following theorem.

\begin{thm}\label{thm3.2}
Suppose that $\f\in{\bf L}^2(\Omega)$ and that ${\bf g}_j\in{\bf H}^{3/2}(\G_j)$ for $1\leq j\leq J$ satisfy \eqref{g-condition} and ${\bf g}_1(S_1)={\bf g}_J(S_1)={\bf 0}$.
For $1\leq i\leq N$, let $\g_i\in\Rbb$ be the number given by
\begin{equation}\label{gamma}
\g_i:=\m\int_{\o_1}^{\o_2}\left(2\l_i\bT_i(\th)\cdot{\widetilde\bT}_i(\th)\right)d\th+\left(\mu+\nu\right)\int_{\o_1}^{\o_2}\left(I_{i}(\th)+J_{i}(\th)\right)d\th,
\end{equation}
where $I_{i}(\th)$ and $J_{i}(\th)$ are defined by
\begin{equation}\label{I12-def}
\begin{aligned}
I_{i}(\th)&:=2\l_i[\bT_i(\th)\cdot\eb_r(\th)][{\widetilde\bT}_i(\th)\cdot\eb_r(\th)],\\
J_{i}(\th)&:=[\pd_\th\bT_i(\th)\cdot\eb_\th(\th)][{\widetilde\bT}_i(\th)\cdot\eb_r(\th)]-[\bT_i(\th)\cdot\eb_r(\th)][\pd_\th{\widetilde\bT}_i(\th)\cdot\eb_\th(\th)].
\end{aligned}
\end{equation}
Then, the coefficients $c_1$ and $c_2$ in \eqref{decomposition} have the following formulae:
\begin{equation}\label{sif-formula}
 c_1=\frac{\mathcal{C}_1(\f,{\bf g})}{\g_1}\quad\mbox{and}\quad
 c_2=\frac{\mathcal{C}_2(\f,{\bf g})+c_1\mathcal{C}_*}{\g_2},
\end{equation}
where
$$
\begin{aligned}
 \mathcal{C}_i(\f,{\bf g}):=&\,\int_\O\f\cdot\left({\widetilde\bPh}_i+\bPs_i\right)d\x\\
 &\,-\sum_{j=1}^J\int_{\G_j}\left(\m\,{\bf g}_j\cdot\frac{\partial}{\partial{\bf n}_j}\left({\widetilde\bPh}_i+\bPs_i\right)+\left(\m+\n\right)\left({\bf g}_j\cdot{\bf n}_j\right)\div\left({\widetilde\bPh}_i+\bPs_i\right)\right)dS,\\
 \mathcal{C}_*:=&\,\sum_{j=2}^{J-1}\int_{\G_j}\left(\m\,\bPh_1\cdot\frac{\partial}{\partial\nb_j}\left({\widetilde\bPh}_2+\bPs_2\right)+\left(\m+\n\right)\left(\bPh_1\cdot\nb_j\right)\div\left({\widetilde\bPh}_2+\bPs_2\right)\right)dS,
\end{aligned}
$$
and $\bPs_i\in{\bf H}^1(\Omega)$ is the weak solution of
\begin{equation}\label{Psi-prob}
\left\{\begin{aligned}
 \mathcal{L}\bPs_i&={\bf 0}&&\mbox{ in }\Omega,\\
 \bPs_i&={\bf 0}&&\mbox{ on }\Gamma_j,\quad j=1,J,\\
 \bPs_i&=-\bPh_i^-&&\mbox{ on }\Gamma_j,\quad j=2,\cdots,J-1.
\end{aligned}\right.
\end{equation}
\end{thm}

\begin{proof}
Set $\v_i:={\widetilde\bPh}_i+\bPs_i$.
For $\epsilon>0$, let $\Omega_\epsilon=\{\x\in\Omega:|\x|>\epsilon\}$.
Multiplying $\v_1$ to the first equation of \eqref{Lame}, integrating the multiplied equation over $\Omega$, and using the integration by parts, we obtain
\begin{align}
\int_\O\f\cdot\v_1\,d\x&=\int_\O\mathcal{L}\u\cdot\v_1\,d\x =\lim_{\eps\rightarrow 0^+}\int_{\O_\eps}\mathcal{L}\u\cdot\v_1\,d\x\nonumber\\
&=\lim_{\eps\rightarrow 0^+}\left(A_\eps+B_\eps\right),\label{f-v1}
\end{align}
where
$$
\begin{aligned}
A_\eps&:=\int_{\partial\O_\eps}\left(-\m\frac{\partial\u}{\partial{\bf n}}\cdot\v_1-\left(\m+\n\right)\div\,\u\left(\v_1\cdot{\bf n}\right)\right)dS,\\
B_\eps&:=\int_{\partial\O_\eps}\left(\m\u\cdot\frac{\partial\v_1}{\partial{\bf n}}
+\left(\m+\n\right)\left(\u\cdot{\bf n}\right)\div\,\v_1\right)dS,
\end{aligned}
$$
and ${\bf n}$ denotes the outward unit normal vector on $\partial\O_\eps$.

First, we consider the quantity $A_\eps$ in \eqref{f-v1}.
For $\eps>0$, let $\Gamma_\epsilon=\{\x\in\O:|\x|=\eps,~\th\in(\o_1,\o_2)\}$ be the arc oriented counter-clockwise with respect to the origin.
By the decomposition \eqref{decomposition} and $\v_1={\bf 0}$ on $\partial\O_\eps\setminus\G_\eps\subset\G$,
\begin{equation}\label{A-eps-eq}
A_\eps=A_{\eps}^{{\rm R}}+A_{\eps}^++A_\eps^-,
\end{equation}
where
$$
\begin{aligned}
A_{\eps}^{{\rm R}}&:=\int_{\G_\eps}\left(-\m\frac{\partial\w}{\partial{\bf n}}\cdot\v_1-\left(\m+\n\right)\div\,\w\left(\v_1\cdot{\bf n}\right)\right)dS,\\
A_{\eps}^+&:=\sum_{i=1}^N c_i\int_{\G_\eps}\left(-\m\frac{\partial\bPh_i}{\partial{\bf n}}\cdot\bPs_1-\left(\m+\n\right)\div\,\bPh_i\left(\bPs_1\cdot{\bf n}\right)\right)dS,\\
A_{\eps}^-&:=\sum_{i=1}^N c_i\int_{\G_\eps}\left(-\m\frac{\partial\bPh_i}{\partial{\bf n}}\cdot{\widetilde\bPh}_1 -\left(\m+\n\right)\div\,\bPh_i\left({\widetilde\bPh}_1\cdot{\bf n}\right)\right)dS.
\end{aligned}
$$
Choose any $p,q\in(1,\infty)$ satisfying $1/p+1/q=1$ and $\l_1 q <1$.
Using H\"{o}lder's inequality, the extended trace theorem and Sobolev embedding ${\rm H}^{2-1/p}(\O)\hookrightarrow {\rm W}^{1+1/p,p}(\O)$ (cf. \cite{GE}), we obtain
\begin{align}
|A_{\eps}^{{\rm R}}|&\leq C_L\|\nab\w\|_{0,p,\G_\eps}\|\v_1\|_{0,q,\G_\eps}\leq C_L\|\w\|_{1,p,\partial\O_\eps}\|\v_1\|_{0,q,\G_\eps}\nonumber\\
&\leq C_L\|\w\|_{1+1/p,p,\O_\eps}\|\v_1\|_{0,q,\G_\eps}\leq C_L\|\w\|_{2-1/p,\O}\|\v_1\|_{0,q,\G_\eps},\label{w-dual-gamma-estimate1}
\end{align}
where $C_L$ denotes a positive constant depending on the Lam\'{e} constants $\m$ and $\n$.
Let $3/2<s<1+\min\{\lambda_1,\kappa_{1}\}$.
Since ${\widetilde\bPh}_i(S_2)={\widetilde\bPh}_i(S_J)={\bf 0}$ and $\|{\widetilde\bPh}_i\|_{s-1/2,\G_j}<\infty$ for $j=2,\ldots,J-1$, there exists a function $\boldsymbol\xi_i\in{\bf H}^{s}(\Omega)$ such that (cf. \cite{GE})
$$
\boldsymbol\xi_i|_{\G_j}=\left\{
\begin{aligned}
&-{\widetilde\bPh}_i|_{\G_j}&&\mbox{ for }j=2,\ldots,J-1,\\
&{\bf 0}&&\mbox{ otherwise}
\end{aligned}\right.
$$
and that
\begin{equation}\label{xi-estimate}
\|\boldsymbol\xi_i\|_{s,\O}\leq C\sum_{j=2}^{J-1}\|{\widetilde\bPh}_i\|_{s-1/2,\G_j}<\infty.
\end{equation}
Letting ${\boldsymbol\eta}_1:=\bPs_1-\boldsymbol\xi_1$, the system \eqref{Psi-prob} becomes
\begin{equation}\label{eta1-eq}
\left\{\begin{aligned}
\mathcal{L}{\boldsymbol\eta}_1&=-\mathcal{L}\boldsymbol\xi_1&&\mbox{ in }\O,\\
{\boldsymbol\eta}_1&={\bf 0}&&\mbox{ on }\G.
\end{aligned}\right.
\end{equation}
By Theorem \ref{thm2.2}, there exists a unique solution ${\boldsymbol\eta}_1\in{\bf H}^s(\Omega)$ to \eqref{eta1-eq} satisfying $\|{\boldsymbol\eta}_1\|_{s,\O}\leq C_L\|\mathcal{L}\boldsymbol\xi_1\|_{s-2,\O}$, and furthermore, by the triangle inequality and \eqref{xi-estimate}, we have
\begin{equation}\label{Psi-estimate}
\|\bPs_1\|_{s,\O}\leq\|{\boldsymbol\eta}_1\|_{s,\O}+\|\boldsymbol\xi_1\|_{s,\O}\leq C_L\|\boldsymbol\xi_1\|_{s,\O}<\infty.
\end{equation}
Since $|{\widetilde\bPh}_1|\leq C_L\eps^{-\lambda_1}$ on $\G_\eps$ and $|\bPs_1|_{\infty}<\infty$ (by \eqref{Psi-estimate}), it follows that
\begin{align}
\|\v_1\|_{0,q,\G_\eps}&\leq\left(\int_{\G_\eps}|{\widetilde\bPh}_1|^q dS\right)^{1/q}+\left(\int_{\G_\eps}|\bPs_1|^q dS\right)^{1/q}\nonumber\\
&\leq\left(\int_{\o_1}^{\o_2}C_L\eps^{1-\l_1 q}d\th\right)^{1/q}+|\bPs_1|_\infty\left(\int_{\G_\eps}dS\right)^{1/q}\nonumber\\
&\leq C_L\eps^{1/q-\l_1}\o^{1/q}+|\bPs_1|_{\infty}\eps^{1/q}\o^{1/q}\rightarrow 0\qquad\mbox{ as }\,\eps\rightarrow 0^+,\label{eps-limit}
\end{align}
where $\o=\o_2-\o_1$.
Applying the result \eqref{eps-limit} to \eqref{w-dual-gamma-estimate1}, it is seen that
\begin{equation}\label{A-w-lim}
\lim_{\eps\rightarrow 0^+}A_{\eps}^{{\rm R}}=0.
\end{equation}
Since $|\nab\bPh_i|\leq C_L\eps^{\l_i-1}$ on $\G_\eps$, we have
\begin{equation}\label{phi-psi-eps}
\begin{aligned}
|A_{\eps}^+|&\leq \sum_{i=1}^N c_i\int_{\o_1}^{\o_2}C_L \eps^{\l_i-1}|\bPs_1(\eps\cos\th,\eps\sin\th)|\,\eps \,d\th\\
&\leq \sum_{i=1}^N c_i \, C_L|\bPs_1|_{\infty}\eps^{\l_i}\o\rightarrow 0\quad\mbox{ as }\,\eps\rightarrow 0^+,
\end{aligned}
\end{equation}
and thus
\begin{equation}\label{A-plus-lim}
\lim_{\eps\rightarrow 0^+}A_{\eps}^+=0.
\end{equation}
By \eqref{A-w-lim} and \eqref{A-plus-lim}, the quantity $A_\eps$ in \eqref{A-eps-eq} has the following convergence:
\begin{equation}\label{A-eps-result}
\lim_{\eps\rightarrow 0^+}A_{\eps}=\lim_{\eps\rightarrow 0^+}A_{\eps}^-.
\end{equation}

Next, we consider the quantity $B_\eps$ in \eqref{f-v1}.
Since ${\bf u}={\bf g}$ on $\partial\O_\eps\setminus\G_\eps\subset\G$, by \eqref{decomposition}, we have
\begin{equation}\label{B-eps-eq}
\begin{aligned}
B_\eps&=B_{\eps,{\bf g}}+B_{\eps}^{{\rm R}}+B_\eps^+ +B_\eps^-,
\end{aligned}
\end{equation}
where
$$
\begin{aligned}
B_{\eps,{\bf g}}&:=\int_{\partial\O_\eps\setminus\G_\eps}\left(\m{\bf g}\cdot\frac{\partial\v_1}{\partial\nb}+\left(\m+\n\right)\left({\bf g}\cdot\nb\right)\div\,\v_1\right)dS,\\
B_{\eps}^{{\rm R}}&:=\int_{\G_\eps}\left(\m\w\cdot\frac{\partial\v_1}{\partial\nb}+\left(\m+\n\right)\left(\w\cdot\nb\right)\div\,\v_1\right)dS,\\
B_{\eps}^+&:=\sum_{i=1}^N c_i\int_{\G_\eps}\left(\m\bPh_i\cdot\frac{\partial\bPs_1}{\partial\nb}+\left(\m+\n\right)\left(\bPh_i\cdot\nb\right)\div\,\bPs_1\right)dS,\\
B_{\eps}^-&:=\sum_{i=1}^N c_i\int_{\G_\eps}\left(\m\bPh_i\cdot\frac{\partial{\widetilde\bPh}_1}{\partial\nb}+\left(\m+\n\right)\left(\bPh_i\cdot\nb\right)\div\,{\widetilde\bPh}_1\right)dS.
\end{aligned}
$$
Since $\pd{\widetilde\bPh}_1/\pd\nb=\l_1\eps^{-\l_1-1}{\widetilde{\bf T}}_1(\th)$ on $\G_\eps$ and $\w={\bf g}_J$ on $\G_J=\{(r\cos\th,r\sin\th)\in\G:\th=\o_2\}$, we obtain
\begin{equation}\label{w-dual-gamma}
\begin{aligned}
\int_{\Gamma_\epsilon}\w\cdot\frac{\partial{\widetilde\bPh}_1}{\partial{\bf n}}dS
&=\lambda_1\epsilon^{-\lambda_1}\int_{\omega_1}^{\omega_2}\w(\epsilon\cos\theta,\epsilon\sin\theta)\cdot{\widetilde{\bf T}}_1\left(\th\right)d\theta\\
&=\lambda_1\epsilon^{-\lambda_1}\bigg({\bf g}_J(\eps\cos\o_2,\eps\sin\o_2)\cdot{\widetilde{\bf A}}_1(\o_2)\\
&\qquad\qquad\;\;\;-\int_{\omega_1}^{\omega_2}\frac{d}{d\theta}\w(\epsilon\cos\theta,\epsilon\sin\theta)\cdot{\widetilde{\bf A}}_1\left(\th\right)d\theta\bigg),
\end{aligned}
\end{equation}
where ${\widetilde{\bf A}}_1(\th):=\int_{\o_1}^\th{\widetilde{\bf T}}_1\left(\sigma\right)d\sigma$.
Since ${\bf g}_J(S_1)={\bf 0}$, by H\"{o}lder's inequality and the Sobolev embedding ${\rm H}^{3/2-1/q}(\G_J)\hookrightarrow {\rm W}^{1,q}(\G_J)$,
\begin{align}
\left|{\bf g}_J(\eps\cos\o_2,\eps\sin\o_2)\right|&=\left|\int_{0}^\eps\frac{d}{dr}{\bf g}_J\left(r\cos\o_2,r\sin\o_2\right)dr\right|\nonumber\\
&\leq C\left(\int_0^\eps dr\right)^{1/p}\|{\bf g}_J\|_{1,q,\G_J}\nonumber\\
&\leq C\eps^{1/p}\|{\bf g}_J\|_{3/2-1/q,\G_J},\label{gj-estimate}
\end{align}
where $p,q\in(1,\infty)$ are H\"{o}lder conjugates with $\l_1 p<1$.
Since ${\widetilde{\bf T}}_1(\th)$ is bounded for $\th\in\left[\o_1,\o_2\right]$, there exists a constant $C_L>0$ such that
\begin{equation}\label{I-estimate}
|{\widetilde{\bf A}}_1(\th)|\leq C_L\qquad\forall\,\th\in\left[\o_1,\o_2\right].
\end{equation}
Using \eqref{gj-estimate} for ${\bf g}_J$ and \eqref{I-estimate} for ${\widetilde{\bf A}}_1$, since $\|r^{-\l_1}\|_{0,p,\G_\eps}\leq C\eps^{1/p-\l_1}$ and $\|\nab\w\|_{0,q,\G_\eps}\leq C\|\w\|_{2-1/q,\O}$ (see \eqref{w-dual-gamma-estimate1}), the quantity \eqref{w-dual-gamma} is estimated as follows:
\begin{align}
&\left|\,\int_{\Gamma_\epsilon}\w\cdot\frac{\partial{\widetilde\bPh}_1}{\partial{\bf n}}dS\,\right|\nonumber\\
&\qquad\leq C_L\eps^{1/p-\l_1}\|{\bf g}_J\|_{3/2-1/q,\G_J}+C_L\int_{\omega_1}^{\omega_2}\epsilon^{-\lambda_1}\left|\nabla\w\left(\epsilon\cos\theta,\epsilon\sin\theta\right)\right|\epsilon \, d\theta\nonumber\\
&\qquad=C_L\eps^{1/p-\l_1}\|{\bf g}_J\|_{3/2-1/q,\G_J}+C_L\int_{\Gamma_\epsilon}r^{-\lambda_1}\left|\nabla\w\right|dS\nonumber\\
&\qquad\leq C_L\eps^{1/p-\l_1}\|{\bf g}_J\|_{3/2-1/q,\G_J}+C_L\|r^{-\l_1}\|_{0,p,\G_\eps}\|\nab\w\|_{0,q,\G_\eps}\nonumber\\
&\qquad\leq C_L\eps^{1/p-\l_1}\left(\|{\bf g}_J\|_{3/2-1/q,\G_J}+\|\w\|_{2-1/q,\O}\right).\label{w-dual-gamma-estimate2}
\end{align}
Similarly, we also have
\begin{equation}\label{w-dual-gamma-estimate3}
\left|\,\int_{\Gamma_\eps}\left(\w\cdot{\bf n}\right)\div\,{\widetilde\bPh}_1 \,dS\,\right|\leq C_L\eps^{1/p-\l_1}\left(\|{\bf g}_J\|_{3/2-1/q,\G_J}+\|\w\|_{2-1/q,\O}\right).
\end{equation}
Since $\w\in{\bf H}^2(\Omega)$, it follows from the extended trace theorem and \eqref{Psi-estimate} that $|\w|_\infty<\infty$ and $\|\nab\bPs_1\|_{0,\G_\eps}\leq C\|\bPs_1\|_{1+\d,\partial\O_\eps}\leq C\|\bPs_1\|_{3/2+\d,\O_\eps}<\infty$ for $0<\d<\l_{1}-1/2$, yielding
\begin{align}
\left|\,\int_{\G_\eps}\left(\m\w\cdot\frac{\partial\bPs_1}{\partial{\bf n}}+\left(\m+\n\right)\left(\w\cdot\nb\right)\div\,\bPs_1\right)dS\,\right|&\leq C_L\|\w\|_{0,\G_\eps}\|\nab\bPs_1\|_{0,\G_\eps}\nbr\\
&\leq C_L\eps^{1/2}\,|\w|_{\infty}\|\bPs_1\|_{3/2+\d,\O_\eps}.\label{w-dual-gamma-estimate4}
\end{align}
By \eqref{w-dual-gamma-estimate2}, \eqref{w-dual-gamma-estimate3} and \eqref{w-dual-gamma-estimate4},
\begin{equation}\label{B-w-estimate}
\begin{aligned}
|B_{\eps}^{{\rm R}}|\leq &\, C_L\eps^{1/p-\l_1}\left(\|{\bf g}_J\|_{3/2-1/q,\G_J}+\|\w\|_{2-1/q,\O}\right)\\
&\,+C_L\eps^{1/2}\,|\w|_\infty\|\bPs_1\|_{3/2+\d,\O_\eps}\rightarrow 0\qquad\mbox{ as }\,\eps\rightarrow 0^+.
\end{aligned}
\end{equation}
%
On the other hand, since $\|\bPh_i\|_{0,\G_\eps}\leq C\eps^{1/2+\l_i}$ and $\|\bPs_1\|_{1,\G_\eps}\leq C\|\bPs_1\|_{s,\O}<\infty$,
\begin{equation}\label{B-eps-plus-estimate}
|B_\eps^+|\leq C_L\sum_{i=1}^N\eps^{1/2+\l_i}\|\bPs_1\|_{s,\O}\rightarrow 0\qquad\mbox{ as }\,\eps\rightarrow 0^+.
\end{equation}
Applying the results \eqref{B-w-estimate} and \eqref{B-eps-plus-estimate} to \eqref{B-eps-eq}, we get
\begin{equation}\label{B-eps-result}
\begin{aligned}
\lim_{\eps\rightarrow 0^+}B_{\eps}&=\sum_{j=1}^J\int_{\G_j}\left(\m{\bf g}_j\cdot\frac{\partial\v_1}{\partial\nb_j}+(\m+\n)({\bf g}_j\cdot\nb_j)\,\div\,\v_1\right)dS+\lim_{\eps\rightarrow 0^+}B_\eps^-.
\end{aligned}
\end{equation}
So, combining \eqref{A-eps-result} and \eqref{B-eps-result} with \eqref{f-v1}, we have
\begin{equation}\label{c1-limit}
\begin{aligned}
\int_\O\f\cdot\v_1 \,d\x=&\,\sum_{j=1}^J\int_{\G_j}\left(\mu{\bf g}_j\cdot\frac{\partial\v_1}{\partial\nb_j}+(\mu+\nu)({\bf g}_j\cdot\nb_j)\,\div\,\v_1\right)dS\\
&\,+\lim_{\eps\rightarrow 0^+}(A_\eps^- +B_\eps^-).
\end{aligned}
\end{equation}
A direct calculation gives
\begin{equation}\label{c1-gamma1-equality}
\begin{aligned}
\lim_{\eps\rightarrow 0^+}(A_\eps^- +B_\eps^-)&=\lim_{\eps\rightarrow 0^+}\epsilon^{\l_i-\l_1}\sum_{i=1}^Nc_i\bigg(\m\int_{\omega_1}^{\omega_2}\left(\l_i+\l_1\right)\left(\bT_i(\theta)\cdot{\widetilde\bT}_1(\theta)\right)d\th\\
&\qquad\qquad\qquad\qquad\quad+(\mu+\nu)\int_{\o_1}^{\o_2}\left(I_{1}^*(\th)+J_{1}^*(\th)\right)d\th\bigg)\\
&=c_1\g_1,
\end{aligned}
\end{equation}
where $\g_1$ is the number defined by \eqref{gamma}, and
$$
\begin{aligned}
I_{1}^*(\th)&:=(\l_i+\l_1)[\bT_i(\th)\cdot\eb_r(\th)][{\widetilde\bT}_1(\th)\cdot\eb_r(\th)],\\
J_{1}^*(\th)&:=[\pd_\th\bT_i(\th)\cdot\eb_\th(\th)][{\widetilde\bT}_1(\th)\cdot\eb_r(\th)]
-[\bT_i(\th)\cdot\eb_r(\th)][\pd_\th{\widetilde\bT}_1(\th)\cdot\eb_\th(\th)].
\end{aligned}
$$
Hence, the desired formula of $c_1$ in \eqref{sif-formula} follows.

\vspace{0.1cm}

We now derive the formula of $c_2$. Consider the boundary value problem of $\u_1:=\w+c_2\bPh_2$ as follows:
\begin{equation}\label{Lame-u1}
\left\{\begin{aligned}
\mathcal{L}\u_1&=\f&&\mbox{ in }\O,\\
\u_1&={\bf g}_j&&\mbox{ on }\G_j,\quad j=1,J,\\
\u_1&={\bf g}_j-c_1\bPh_1|_{\G_j}&&\mbox{ on }\G_j,\quad j=2,\cdots,J-1.
\end{aligned}\right.
\end{equation}
Multiplying $\v_2$ to the first equation of \eqref{Lame-u1}, integrating the multiplied equation over $\O$, and using the integration by parts,
$$
\begin{aligned}
\int_\O\f\cdot\v_2\,d\x&=\lim_{\eps\rightarrow 0^+}\int_{\O_\eps}\mathcal{L}\u_1\cdot\v_2\,d\x\\
&=\lim_{\eps\rightarrow 0^+}(C_\eps+D_\eps),
\end{aligned}
$$
where
$$
\begin{aligned}
C_\eps&:=-\int_{\partial\O_\eps}\left(\m\frac{\partial\u_1}{\partial{\bf n}}\cdot\v_2+(\m+\n)\,\div\,\u_1(\v_2\cdot{\bf n})\right)dS,\\
D_\eps&:=\int_{\partial\O_\eps}\left(\m\u_1\cdot\frac{\partial\v_2}{\partial{\bf n}}
+(\m+\n)(\u_1\cdot{\bf n})\,\div\,\v_2\right)dS.
\end{aligned}
$$
As in the derivation of \eqref{c1-limit} from \eqref{A-eps-result} and \eqref{B-eps-result}, we obtain
\begin{equation}\label{C2-eps}
\begin{aligned}
\int_\O\f\cdot\v_2 \,d\x=&\,\sum_{j=1}^J\int_{\G_j}\left(\mu{\bf g}_j\cdot\frac{\partial\v_2}{\partial\nb_j}+(\mu+\nu)({\bf g}_j\cdot\nb_j)\,\div\,\v_2\right)dS\\
&\,-c_1\sum_{j=2}^{J-1}\int_{\G_j}\left(\mu{\bPh}_1\cdot\frac{\partial\v_2}{\partial\nb_j}+(\mu+\nu)({\bPh}_1\cdot\nb_j)\,\div\,\v_2\right)dS\\
&\,+\lim_{\eps\rightarrow 0^+}\,(C_\eps^-+D_\eps^-),
\end{aligned}
\end{equation}
where
$$
\begin{aligned}
C_{\eps}^-&:=-c_2\int_{\G_\eps}\left(\m\frac{\partial\bPh_2}{\partial{\bf n}}\cdot{\widetilde\bPh}_2 +(\m+\n)\,\div\,\bPh_2({\widetilde\bPh}_2\cdot{\bf n})\right)dS,\\
D_{\eps}^-&:=c_2\int_{\G_\eps}\left(\m\bPh_2\cdot\frac{\partial{\widetilde\bPh}_2}{\partial\nb}+(\m+\n)(\bPh_2\cdot\nb)\,\div\,{\widetilde\bPh}_2\right)dS.
\end{aligned}
$$
As shown in \eqref{c1-gamma1-equality}, the equality \eqref{C2-eps} becomes
$$
\mathcal{C}_2(\f,{\bf g})+c_1\mathcal{C}_*=c_2\g_2,
$$
which is the desired formula of $c_2$ in \eqref{sif-formula}.
\end{proof}

\vspace{0.1cm}

We next show the a priori estimates of the coefficients $c_i$.

\begin{thm}\label{thm3.3}
Assume that $\f\in{\bf L}^2(\Omega)$ and that ${\bf g}_j\in{\bf H}^{3/2}(\G_j)$ for $1\leq j\leq J$ satisfy \eqref{g-condition} and ${\bf g}_1(S_1)={\bf g}_J(S_1)={\bf 0}$. Then there exists a positive constant $C$ independent of the Lam\'{e} constants $\m$ and $\nu$ such that the coefficients $c_i$, formulated by \eqref{sif-formula}, have the following a priori estimate:
\begin{equation}\label{Lame-sif-estimate}
|c_i|\leq C\bigg(\mu^{-1}\|\f\|_{0,\Omega}+\sum_{j=1}^J\|{\bf g}_j\|_{3/2,\G_j}\bigg).
\end{equation}
\end{thm}

\begin{proof}
First, we consider the estimate of $\C_i(\f,{\bf g})$ in \eqref{sif-formula}.
The quantity $\C_i(\f,{\bf g})$ is rewritten as
$$
\C_i(\f,{\bf g})=F_i-\sum_{j=1}^J(G_{i,j}^-+G_{i,j}^+),
$$
where
$$
\begin{aligned}
F_i&:=\int_\O\f\cdot\left({\widetilde\bPh}_i+\bPs_i\right)d{\bf x},\\
G_{i,j}^-&:=\int_{\G_j}\bigg(\m{\bf g}_j\cdot\frac{\pd{\widetilde\bPh}_i}{\pd\nb_j}+\left(\m+\n\right)\left({\bf g}_j\cdot\nb_j\right)\div\,{\widetilde\bPh}_i\bigg)dS,\\
G_{i,j}^+&:=\int_{\G_j}\bigg(\m{\bf g}_j\cdot\frac{\pd\bPs_i}{\pd\nb_j}+\left(\m+\n\right)\left({\bf g}_j\cdot\nb_j\right)\div\,\bPs_i\bigg)dS.
\end{aligned}
$$
Considering the form of ${\widetilde\bPh}_i$ in \eqref{Dual}, it is directly observed that
\begin{equation}\label{dual_sing_l2}
\|{\widetilde\bPh}_i\|_{0,\O}^2\leq C\int_0^{r_*}r^{-2\l_i+1}dr\leq C,
\end{equation}
where $r_*:=\mbox{diam}(S_1,\O)$.
In addition, Theorem \ref{thm2.2} gives that the solution $\bPs_i$ of \eqref{Psi-prob} has the following a priori estimate: for $1<s<1+\min\{\lambda_1,\kappa_1\}$,
\begin{align}
\m\|\bPs_i\|_{s,\O}+(\m+\n)\|\div\,\bPs_i\|_{s-1,\O}&\leq C\m\sum_{j=2}^{J-1}\|{\widetilde\bPh}_i\|_{s-1/2,\G_j}\nbr\\
&\leq C\m.\label{psi_estimate}
\end{align}
By \eqref{dual_sing_l2} and \eqref{psi_estimate}, we have
$$
\begin{aligned}
|F_i|&\leq C\|\f\|_{0,\O}\left(\|{\widetilde\bPh}_i\|_{0,\O}+\|\bPs_i\|_{0,\O}\right)\\
&\leq C\|\f\|_{0,\O}.
\end{aligned}
$$
Now, we consider the quantity $G_{i,j}^-$.
A direct calculation gives
\begin{equation}\label{div_dual_sing}
\begin{aligned}
\left(\m+\n\right)|\div\,{\widetilde\bPh}_i|=4\m\l_i r^{-\l_i-1}|K_i(\th)|\leq C\m r^{-\l_i-1},
\end{aligned}
\end{equation}
where
$$
K_i(\th):=\left\{\begin{aligned}
&\sin[(1+\l_1)(\th-\ol\o)]&&\mbox{ for }\,i=1,\\
&\cos[(1+\l_2)(\th-\ol\o)]&&\mbox{ for }\,i=2.
\end{aligned}\right.
$$
By \eqref{div_dual_sing}, $|\pd{\widetilde\bPh}_i/\pd\nb_j|\leq Cr^{-\l_i-1}$ and ${\bf g}_1(S_1)={\bf 0}$, we obtain
$$
\begin{aligned}
|G_{i,1}^-|&\leq C\m\int_0^{r_*}|{\bf g}_1(r\cos\omega_1,r\sin\omega_2)|\,r^{-\l_i-1}dr\\
&=C\m\int_0^{r_*}\left|\frac{1}{r}\int_0^r \frac{d}{dr}{\bf g}_1(r_1\cos\o_1,r_1\sin\o_1)\,dr_1\right| r^{-\l_i}dr\\
&\leq C\m\left(\int_0^{r_*}\left|\frac{1}{r}\int_0^r\frac{d}{dr}{\bf g}_1(r_1\cos\o_1,r_1\sin\o_1)\,dr_1\right|^p dr\right)^{1/p}\left(\int_0^{r_*}r^{-\l_i q}\,dr\right)^{1/q}\\
&\leq C\m\left(\int_0^{r_*}\left|\frac{d}{dr}{\bf g}_1(r\cos\o_1,r\sin\o_1)\right|^p dr\right)^{1/p}\qquad\mbox{(by Hardy's inequality)}\\
&\leq C\m\|{\bf g}_1\|_{1,p,\G_1}\leq C\m\|{\bf g}_1\|_{3/2,\G_1},
\end{aligned}
$$
where $p=1/\d$ and $q=(1-\d)^{-1}$ for $0<\d<1-\l_i$.
Similarly, we also have
$$
|G_{i,J}^-|\leq C\m\|{\bf g}_J\|_{3/2,\G_J}.
$$
Using \eqref{div_dual_sing}, for $j=2,3,\ldots,J-1$,
$$
\begin{aligned}
|G_{i,j}^-|&=\left|\int_{\G_j}\left(\m{\bf g}_j\cdot\frac{\pd{\widetilde\bPh}_i}{\pd\nb_j}+\left(\m+\n\right)\left({\bf g}_j\cdot\nb_j\right)\div\,{\widetilde\bPh}_i\right)dS\,\right|\\
&\leq C\m\int_{\G_j}|{\bf g}_j|r^{-\l_i-1}dS\leq C\m\|{\bf g}_j\|_{0,\G_j}\|r^{-\l_i-1}\|_{0,\G_j}\\
&\leq C\m\|{\bf g}_j\|_{3/2,\G_j}.
\end{aligned}
$$
By the Cauchy-Schwarz inequality and inequality \eqref{psi_estimate},
\begin{equation}\label{G_plus_ij}
\begin{aligned}
|G_{i,j}^+|&\leq C\|{\bf g}_j\|_{0,\G_j}\left(\m\|\bPs_i\|_{1,\G_j}+(\m+\n)\|\div\,\bPs_i\|_{0,\G_j}\right)\\
&\leq C\|{\bf g}_j\|_{0,\G_j}\left(\m\|\bPs_i\|_{s,\O}+(\m+\n)\|\div\,\bPs_i\|_{s-1,\O}\right)\quad\left(3/2<s<1+\min\{\lambda_1,\kappa_1\}\right)\\
&\leq C\m\|{\bf g}_j\|_{3/2,\G_j}.
\end{aligned}
\end{equation}
Hence, we have
\begin{align}
|\C_i(\f,{\bf g})|&\leq |F_i|+\sum_{j=1}^J(|G_{i,j}^-|+|G_{i,j}^+|)\nbr\\
&\leq C\|\f\|_{0,\O}+C\m\sum_{j=1}^J\|{\bf g}_j\|_{3/2,\G_j}.\label{Lame-Ci-estimate}
\end{align}

Next, we consider the estimate of $\C_*$ in \eqref{sif-formula}.
Since $|\bPh_1|\leq Cr^{\l_1}$ and $|\pd{\widetilde\bPh}_2/\pd\nb_j|\leq Cr^{-\l_2-1}$, it follows from \eqref{psi_estimate}, \eqref{div_dual_sing} and the extended trace theorem that
\begin{align}
|\C_*|&=\Bigg|\,\sum_{j=2}^{J-1}\int_{\G_j}\left(\m\bPh_1\cdot\frac{\pd\v_2}{\pd\nb_j}+\left(\m+\n\right)\left(\bPh_1\cdot\nb\right)\div\,\v_2\right)dS \,\Bigg|\nbr\\
&\leq C\m\sum_{j=2}^{J-1}\int_{\G_j}r^{\l_1-\l_2-1}dS+C\sum_{j=2}^{J-1}\|\bPh_1\|_{0,\G_j}\left(\m\|\bPs_2\|_{1,\G_j}+\left(\m+\n\right)\|\div\,\bPs_2\|_{0,\G_j}\right)\nbr\\
&\leq C\m+C\left(\m\|\bPs_2\|_{s,\O}+\left(\m+\n\right)\|\div\,\bPs_2\|_{s-1,\O}\right)\qquad\left(3/2<s<1+\min\{\lambda_1,\kappa_1\}\right)\nbr\\
&\leq C\m.\label{Lame-Cstar-estimate}
\end{align}
Since $|\gamma_i|^{-1}\leq C\mu^{-1}$,  we have from \eqref{Lame-Ci-estimate} that
$$
\begin{aligned}
|c_1|&\leq |\g_1|^{-1}|\C_1(\f,{\bf g})|\\
&\leq C\bigg(\m^{-1}\|\f\|_{0,\O}+\sum_{j=1}^J\|{\bf g}_j\|_{3/2,\G_j}\bigg),
\end{aligned}
$$
and furthermore, by \eqref{Lame-Cstar-estimate},
$$
\begin{aligned}
|c_2|&\leq|\g_2|^{-1}(|\C_2(\f,{\bf g})|+|c_1||\C_*|)\\
&\leq C\bigg(\m^{-1}\|\f\|_{0,\O}+\sum_{j=1}^J\|{\bf g}_j\|_{3/2,\G_j}\bigg).
\end{aligned}
$$
Therefore, the desired estimate \eqref{Lame-sif-estimate} can be derived.
\end{proof}

\vspace{0.7cm}

\section{Generalized corner singularities of the Stokes equations}\label{sec4}

\vspace{0.2cm}

In this section, we show the generalized corner singularity expansion of the Stokes equations with inhomogeneous boundary condition \eqref{Stokes} (cf. \cite{CKT}).
To investigate the convergence of both singular and regular parts regarding the penalized Lam\'{e} system \eqref{Lame_eps}, we here discuss corner singularity results on both the inhomogeneous boundary condition and non-divergence-free property of velocity field, and in particular, this generalization does not use a cutoff function either.
Based on the inhomogeneous boundary condition and non-divergence-free property, we try to derive the explicit formulae of coefficients in singular part to the Stokes equations, which will be the limits of the coefficients $c_i\,(1\leq i\leq N)$, given by \eqref{sif-formula}, when the second Lam\'{e} constant $\nu$ goes to infinity.

\vspace{0.1cm}

We introduce singular functions of the Stokes operator with Dirichlet boundary condition.
As derived in \cite[Section 5.1]{KMRS}, its eigenvalues, denoted by $\kappa_i$, are roots of the trigonometric equation: $\sin^2(\kappa_i\,\o)-\kappa_i^2\sin^2\o=0$, where the first three positive eigenvalues are real and ordered as
\begin{equation}\label{Stokes-eigenvalue}
\begin{aligned}
\mbox{\bf Case 1. }~&1/2<\kappa_1<\pi/\o<\kappa_2=1<\kappa_3<2\pi/\o&&~\mbox{ if }~\o\in(\pi,\o_*],\\
\mbox{\bf Case 2. }~&1/2<\kappa_1<\pi/\o<\kappa_2<\kappa_3=1<2\pi/\o&&~\mbox{ if }~\o\in(\o_*,2\pi),
\end{aligned}
\end{equation}
with $(0,2\pi)\ni\o_*\approx 1.4303\pi$ being the positive number satisfying $\tan\o_*=\o_*$.
Denote by $M$ the number of singular functions not in $\H^2(\O)$, i.e., $M=1$ if $\o\in(\pi,\o_*]$, and $M=2$ if $\o\in(\o_*,2\pi)$.
So, we remark that $1/2<\kappa_i<1$ for $1\leq i\leq M$.
Throughout this paper, it is assumed that $M=2$ without loss of generality.
Corresponding to each eigenvalue $\kappa_i$, the singular functions $\bPh_i^{\rm s}$ for velocity and $\phi_i^{\rm s}$ for pressure are defined as
\begin{equation}\label{Stokes-Singular}
\bPh_i^{\rm s}(r,\th)=\mu^{-1}r^{\kappa_i}\bT_i^{\rm s}(\th)\quad\mbox{and}\quad\phi_i^{\rm s}(r,\th)=r^{\kappa_i-1}\xi_i^{\rm s}(\th),
\end{equation}
where $\bT_i^{\rm s}(\th)$ and $\xi_i^{\rm s}(\th)$ are trigonometric eigenfunctions corresponding to $\kappa_i$, given by
\begin{equation}\label{Stokes-T-def}
\bT_i^{\rm s}(\th)=A_i^{\rm s}(\th-\ol\o)\,\eb_r(\th)+B_i^{\rm s}(\th-\ol\o)\,\eb_\th(\th),\qquad\xi_i^{\rm s}(\th)=C_i^{\rm s}(\th-\ol\o),
\end{equation}
and
$$
\begin{aligned}
\begin{pmatrix}
A^{\rm s}_1(\th)\\B^{\rm s}_1(\th)\\C^{\rm s}_1(\th)
\end{pmatrix}
:=&\,-\begin{pmatrix}
(1-\kappa_1)\sin[(1-\kappa_1)\th]\\
(1+\kappa_1)\cos[(1-\kappa_1)\th]\\
-4\kappa_1\sin[(1-\kappa_1)\th]
\end{pmatrix}\\
&\, +\left(\cos(\kappa_1\omega)+\kappa_1\cos\omega\right)\begin{pmatrix}
\sin[(1+\kappa_1)\th]\\
\cos[(1+\kappa_1)\th]\\
0
\end{pmatrix},\\
\begin{pmatrix}
A^{\rm s}_2(\th)\\B^{\rm s}_2(\th)\\C^{\rm s}_2(\th)
\end{pmatrix}:=&\,-\begin{pmatrix}
(1-\kappa_2)\cos[(1-\kappa_2)\th]\\
-(1+\kappa_2)\sin[(1-\kappa_2)\th]\\
-4\kappa_2\cos[(1-\kappa_2)\th]
\end{pmatrix}\\
&\, +\left(\cos(\kappa_2\omega)-\kappa_2\cos\omega\right)\begin{pmatrix}
\cos[(1+\kappa_2)\th]\\
-\sin[(1+\kappa_2)\th]\\
0
\end{pmatrix}.
\end{aligned}
$$

\vspace{0.1cm}

Now, we state the corner singularity expansion for the Stokes problem \eqref{Stokes} (cf. \cite{CKT}).

\begin{thm}\label{thm4.1}
If $\f\in{\bf L}^2(\O),$ $\zeta\in{\rm H}^{1}(\O)$ with $\zeta(S_1)=0$ and ${\bf g}_j\in{\bf H}^{3/2}(\Gamma_j)$ for $1\leq j\leq J$ with \eqref{g-condition}, then there exists a unique solution pair $[\u^{{\rm s}},\,p^{{\rm s}}]\in{\bf H}^1(\O)\times{\rm L}_0^2(\O)$ to the Stokes problem \eqref{Stokes} that can be written as follows: there are numbers $c^{\rm s}_i\in\mathbb{R}$ $(1\leq i\leq M)$ and a pair $\left[\w^{\rm s},\,\sigma^{\rm s}\right]\in{\bf H}^2(\O)\times{\rm H}^{1}(\O)$ satisfying
\begin{equation}\label{Stokes-decomposition}
\left[\u^{{\rm s}},\,p^{{\rm s}}\right]=\left[\w^{\rm s},\,\sigma^{\rm s}\right]+\sum_{i=1}^{M}c^{\rm s}_i\left[\bPh^{\rm s}_i,\,\phi^{\rm s}_i\right].
\end{equation}

\end{thm}

\begin{proof}
From the extended trace theorem \cite{GE}, there exists a vector function ${\bf z}\in{\bf H}^2(\O)$ satisfying ${\bf z}|_{\Gamma_j}={\bf g}_j$ for $1\leq j\leq J$ and
$$
\|{\bf z}\|_{2,\O}\leq C\sum_{j=1}^{J}\|{\bf g}_j\|_{3/2,\Gamma_j}.
$$
Setting $\u_0:=\u^{{\rm s}}-{\bf z}$, the system \eqref{Stokes} can be rewritten as
\begin{equation}\label{mod-Stokes}
\left\{\begin{aligned}
-\mu\Delta{\bf u}_0+\nabla p&={\bf f}+\mu\Delta{\bf z}&&\mbox{ in }\Omega,\\
\div\,\u_0&=\zeta-\div\,{\bf z}&&\mbox{ in }\Omega,\\
\u_0&={\bf 0}&&\mbox{ on }\Gamma.
\end{aligned}\right.
\end{equation}
Since $\f+\mu\Delta\z\in{\bf L}^2(\O)$ and $\zeta-\div\,\z\in\bar{\rm H}^1(\Omega)$, by \cite[Theorem 2.1]{CKT}, there exist $c^{\rm s}_i\in\mathbb{R}$ for $1\leq i\leq M$ and a pair $\left[\u_{\rm R},\,p_{\rm R}\right]\in{\bf H}^2(\O)\times{\rm H}^1(\O)$ such that
$$
\begin{aligned}
\left[\u_0,\,p\right]&=\left[\u_{\rm R},\,p_{\rm R}\right]+\sum_{i=1}^{M} c^{\rm s}_i\ch\left[\bPh^{\rm s}_i,\,\phi^{\rm s}_i\right]\\
&=\left(\left[\u_{\rm R},\,p_{\rm R}\right]+\sum_{i=1}^{M} c^{\rm s}_i(\ch-1)\left[\bPh^{\rm s}_i,\,\phi^{\rm s}_i\right]\right)+\sum_{i=1}^{M} c^{\rm s}_i\left[\bPh^{\rm s}_i,\,\phi^{\rm s}_i\right],
\end{aligned}
$$
where $\ch=\ch(r)$ is the cutoff function defined in the proof of Theorem \ref{thm3.1}.
Since the value of $(\ch-1)\left[\bPh^{\rm s}_i,\,\phi^{\rm s}_i\right]$ is zero near the re-entrant corner $S_1$, it belongs to $\H^2(\O)\times{\rm H}^1(\O)$ so that the desired decomposition \eqref{Stokes-decomposition} is obtained by letting $\w^{\rm s}:=\u_{\rm R}+\sum_{i=1}^{M}c^{\rm s}_i(\ch-1)\bPh^{\rm s}_i+\z\in{\bf H}^2(\O)$ and $\sigma^{\rm s}:=p_{\rm R}+\sum_{i=1}^{M}c^{\rm s}_i(\ch-1)\phi^{\rm s}_i\in{\rm H}^1(\O)$.
\end{proof}

\vspace{0.1cm}

We next formulate the coefficients $c^{\rm s}_i$ in \eqref{Stokes-decomposition}.
To derive their formulae, we need to define the dual singular functions ${\widetilde\bPh}^{\rm s}_i$ and ${\widetilde\phi}^{\rm s}_i$ of the Stokes operator as
\begin{equation}\label{Stokes-Dual}
{\widetilde\bPh}^{\rm s}_i(r,\th)=\mu^{-1}r^{-\kappa_i}{\widetilde\bT}^{\rm s}_i(\th)\quad\mbox{and}\quad{\widetilde\phi}^{\rm s}_i(r,\th)=r^{-\kappa_i-1}{\widetilde\xi}^{\rm s}_i(\th),
\end{equation}
where ${\widetilde\bT}^{\rm s}_i(\th)$ and ${\widetilde\xi}^{\rm s}_i(\th)$ are obtained by replacing $\kappa_i$ with $-\kappa_i$ in the eigenfunctions $\bT^{\rm s}_i(\th)$ and $\xi^{\rm s}_i(\th)$ given by \eqref{Stokes-T-def}, respectively.
Using ${\widetilde\bPh}^{\rm s}_i$ and ${\widetilde\phi}^{\rm s}_i$ in \eqref{Stokes-Dual}, we give our desired formulae of $c^{\rm s}_i$ in the following theorem.

\begin{thm}\label{thm4.2}
Assume that $\f\in{\bf L}^2(\Omega)$, $\zeta\in{\rm H}^1(\Omega)$ and ${\bf g}_j\in{\bf H}^{3/2}(\G_j)$ for $1\leq j\leq J$ with \eqref{g-condition} and ${\bf g}_1(S_1)={\bf g}_J(S_1)={\bf 0}$.
For $1\leq i\leq M$, let $\g^{\rm s}_i\in\Rbb$ be the number given by
\begin{equation}\label{Stokes-gamma}
\g^{\rm s}_i:=\int_{\o_1}^{\o_2}\left(2\kappa_i\bT^{\rm s}_i(\th)\cdot{\widetilde\bT}^{\rm s}_i(\th)-\xi_i^{\rm s}(\th)\,[{\widetilde\bT}^{\rm s}_i(\th)\cdot\eb_r(\th)]+[\bT_i^{\rm s}(\th)\cdot\eb_r(\th)]\,{\widetilde\xi}^{\rm s}_i(\th)\right)d\th.
\end{equation}
Then, the coefficients $c^{\rm s}_1$ and $c^{\rm s}_2$ in \eqref{Stokes-decomposition} have the following formulae:
\begin{equation}\label{Stokes-sif-formula}
 c^{\rm s}_1=\frac{\mathcal{C}^{\rm s}_1(\f,\zeta,{\bf g})}{\g^{\rm s}_1}\quad\mbox{and}\quad
 c^{\rm s}_2=\frac{\mathcal{C}^{\rm s}_2(\f,\zeta,{\bf g})+c^{\rm s}_1\mathcal{C}^{\rm s}_*}{\g^{\rm s}_2},
\end{equation}
where
\begin{equation}\label{Stokes-C}
\begin{aligned}
 \mathcal{C}^{\rm s}_i(\f,\zeta,{\bf g}):=&\,\int_\O\left(\f\cdot\left(\mu{\widetilde\bPh}^{\rm s}_i+\bPs^{\rm s}_i\right)-\zeta\left(\mu{\widetilde\ph}^{\rm s}_i+\ps^{\rm s}_i\right)\right)d\x\\
 &\,-\sum_{j=1}^J\int_{\G_j}\left(\m\,{\bf g}_j\cdot\frac{\partial}{\partial{\bf n}_j}\left(\mu{\widetilde\bPh}^{\rm s}_i+\bPs^{\rm s}_i\right)-\left({\bf g}_j\cdot{\bf n}_j\right)\left(\mu{\widetilde\ph}^{\rm s}_i+\ps^{\rm s}_i\right)\right)dS,\\
 \mathcal{C}^{\rm s}_*:=&\,\sum_{j=2}^{J-1}\int_{\G_j}\left(\m\bPh^{\rm s}_1\cdot\frac{\partial}{\partial\nb_j}\left(\mu{\widetilde\bPh}^{\rm s}_2+\bPs^{\rm s}_2\right)-\left(\bPh^{\rm s}_1\cdot\nb_j\right)\left(\mu{\widetilde\ph}^{\rm s}_2+\ps^{\rm s}_2\right)\right)dS,
\end{aligned}
\end{equation}
and $\left[\bPs^{\rm s}_i,\,\psi^{\rm s}_i\right]\in{\bf H}^1(\Omega)\times{\rm L}_0^2(\Omega)$ is the weak solution pair of
\begin{equation}\label{Stokes-ps}
\left\{\begin{aligned}
 -\mu\Delta\bPs^{\rm s}_i+\nabla\psi^{\rm s}_i&={\bf 0}&&\mbox{ in }\Omega,\\
 \div\,\bPs^{\rm s}_i&=0&&\mbox{ in }\Omega,\\
 \bPs^{\rm s}_i&={\bf 0}&&\mbox{ on }\Gamma_j,\quad j=1,J,\\
 \bPs^{\rm s}_i&=-\mu{\widetilde\bPh}^{\rm s}_i&&\mbox{ on }\Gamma_j,\quad j=2,\cdots,J-1.
\end{aligned}\right.
\end{equation}
\end{thm}

\begin{proof}
Set $\v^{\rm s}_i:=\mu{\widetilde\bPh}^{\rm s}_i+\bPs^{\rm s}_i$ and $q^{\rm s}_i:=\mu{\widetilde\ph}^{\rm s}_i+\ps^{\rm s}_i$. Using the integration by parts, the system \eqref{Stokes} gives
\begin{align}
\int_\Omega\left(\f\cdot\v^{\rm s}_1-\zeta q^{\rm s}_1\right)d\x&=\lim_{\eps\rightarrow 0^+}\int_{\Omega_\eps}\left(\left(-\mu\Delta\u^{{\rm s}}+\nabla p^{{\rm s}}\right)\cdot\v^{\rm s}_1-\div\,\u^{{\rm s}}\,q^{\rm s}_1\right)d\x\nbr\\
&=\lim_{\eps\rightarrow 0^+}\left(A_\eps+B_\eps\right),\label{f-A-B-equality}
\end{align}
where $\Omega_\eps=\{\x\in\Omega:\left|\x\right|>\eps\}$ for $\eps>0,$ and
$$
\begin{aligned}
A_\eps&:=\int_{\partial\Omega_\eps}\left(-\mu\frac{\partial\u^{{\rm s}}}{\partial\nb}\cdot\v^{\rm s}_1+p^{{\rm s}}\left(\v^{\rm s}_1\cdot\nb\right)\right)dS,\\
B_\eps&:=\int_{\partial\Omega_\eps}\left(\mu\u^{{\rm s}}\cdot\frac{\partial\v^{\rm s}_1}{\partial\nb}-\left(\u^{{\rm s}}\cdot\nb\right)q^{\rm s}_1\right)dS.
\end{aligned}
$$

We first consider $A_\eps$ in \eqref{f-A-B-equality}.
Since $\left[\u^{{\rm s}},\,p^{{\rm s}}\right]=\left[\w^{\rm s},\,\sigma^{\rm s}\right]+\sum_{i=1}^{M} c^{\rm s}_i\left[\bPh^{\rm s}_i,\,\ph^{\rm s}_i\right]$ in $\Omega$ and $\v^{\rm s}_1={\bf 0}$ on $\partial\Omega_\eps\setminus\Gamma_\eps\subset\Gamma$, we have
\begin{equation}\label{Stokes-A-eps}
A_\eps=A_\eps^{{\rm R}}+A_\eps^++A_\eps^-,
\end{equation}
where
$$
\begin{aligned}
A_\eps^{{\rm R}}&:=\int_{\Gamma_\eps}\left(-\mu\frac{\partial\w^{\rm s}}{\partial\nb}\cdot\v^{\rm s}_1+\sigma^{\rm s}\left(\v^{\rm s}_1\cdot\nb\right)\right)dS,\\
A_\eps^+&:=\sum_{i=1}^{M} c^{\rm s}_i\int_{\Gamma_\eps}\left(-\mu\frac{\partial\bPh^{\rm s}_i}{\partial\nb}\cdot\bPs^{\rm s}_1+\ph^{\rm s}_i\left(\bPs^{\rm s}_1\cdot\nb\right)\right)dS,\\
A_\eps^-&:=\sum_{i=1}^{M} c^{\rm s}_i\int_{\Gamma_\eps}\left(-\mu\frac{\partial\bPh^{\rm s}_i}{\partial\nb}\cdot\mu{\widetilde\bPh}^{\rm s}_1+\ph^{\rm s}_i\left(\mu{\widetilde\bPh}^{\rm s}_1\cdot\nb\right)\right)dS,
\end{aligned}
$$
and $\G_\eps=\{\x\in\Omega:\left|\x\right|=\eps,\,\theta\in\left(\o_1,\o_2\right)\}$ is the arc oriented counter-clockwise with respect to the origin.
Using H\"{o}lder's inequality with $p,q\in(1,\infty)$ satisfying $1/p+1/q=1$ and $\kappa_{1}q<1$, the extended trace theorem and Sobolev embedding ${\rm H}^{2-1/p}(\Omega)\hookrightarrow {\rm W}^{1+1/p,p}(\Omega)$,
\begin{align}
\left|A_\eps^{{\rm R}}\right|&\leq C\left(\mu\|\nabla\w^{\rm s}\|_{0,p,\Gamma_\eps}+\|\sigma^{\rm s}\|_{0,p,\Gamma_\eps}\right)\|\v^{\rm s}_1\|_{0,q,\Gamma_\eps}\nonumber\\
&\leq C\left(\mu\|\w^{\rm s}\|_{1,p,\partial\Omega_\eps}+\|\sigma^{\rm s}\|_{0,p,\partial\Omega_\eps}\right)\|\v^{\rm s}_1\|_{0,q,\Gamma_\eps}\nonumber\\
&\leq C\left(\mu\|\w^{\rm s}\|_{1+1/p,p,\Omega_\eps}+\|\sigma^{\rm s}\|_{1/p,p,\Omega_\eps}\right)\|\v^{\rm s}_1\|_{0,q,\Gamma_\eps}\nonumber\\
&\leq C\left(\mu\|\w^{\rm s}\|_{2-1/p,\Omega}+\|\sigma^{\rm s}\|_{1-1/p,\Omega}\right)\|\v^{\rm s}_1\|_{0,q,\Gamma_\eps}.\label{A-eps-R-estimate1}
\end{align}
Let $3/2<l<1+\kappa_{1}$.
Since ${\widetilde\bPh}^{\rm s}_1(S_2)={\widetilde\bPh}^{\rm s}_1(S_J)={\bf 0}$ and $\|{\widetilde\bPh}^{\rm s}_1\|_{l-1/2,\G_j}<\infty$ for $j=2,\cdots,J-1$, there exists a vector function $\boldsymbol\xi^{\rm s}_1\in{\bf H}^l(\Omega)$ such that
$$
\boldsymbol\xi^{\rm s}_1|_{\G_j}=\left\{
\begin{aligned}
&-\mu{\widetilde\bPh}^{\rm s}_1|_{\G_j}&&\mbox{ for }j=2,\ldots,J-1,\\
&{\bf 0}&&\mbox{ otherwise}
\end{aligned}\right.
$$
and that
\begin{equation}\label{xi-s-estimate}
\|\boldsymbol\xi^{\rm s}_1\|_{l,\O}\leq C\mu\sum_{j=2}^{J-1}\|{\widetilde\bPh}^{\rm s}_1\|_{l-1/2,\G_j}<\infty.
\end{equation}
With ${\boldsymbol\eta}^{\rm s}_1:=\bPs^{\rm s}_1-\boldsymbol\xi^{\rm s}_1$, the problem \eqref{Stokes-ps} becomes
$$
\left\{\begin{aligned}
-\mu\Delta{\boldsymbol\eta}^{\rm s}_1+\nabla\psi^{\rm s}_1&=\mu\Delta{\boldsymbol\xi}^{\rm s}_1&&\mbox{ in }\Omega,\\
\div\,{\boldsymbol\eta}^{\rm s}_1&=-\div\,{\boldsymbol\xi}^{\rm s}_1&&\mbox{ in }\Omega,\\
{\boldsymbol\eta}^{\rm s}_1&={\bf 0}&&\mbox{ on }\Gamma.
\end{aligned}\right.
$$
By \cite[Theorem 2.1]{CKT}, the existence and uniqueness of $\left[{\boldsymbol\eta}^{\rm s}_1,\,\ps^{\rm s}_1\right]\in{\bf H}_0^l(\Omega)\times\bar{\rm H}^{l-1}(\Omega)$ are guaranteed with
\begin{equation}\label{eta-s-estimate}
\mu\|{\boldsymbol\eta}^{\rm s}_1\|_{l,\O}+\|\ps^{\rm s}_1\|_{l-1,\O}\leq C\mu\left(\|\Delta{\boldsymbol\xi}^{\rm s}_1\|_{l-2,\O}+\|\div\,{\boldsymbol\xi}^{\rm s}_1\|_{l-1,\O}\right).
\end{equation}
Using the triangle inequality, \eqref{eta-s-estimate} and \eqref{xi-s-estimate},
\begin{align}
\mu\|\bPs^{\rm s}_1\|_{l,\O}+\|\ps^{\rm s}_1\|_{l-1,\Omega}&\leq\mu\|{\boldsymbol\eta}^{\rm s}_1\|_{l,\O}+\mu\|\boldsymbol\xi^{\rm s}_1\|_{l,\O}+\|\ps^{\rm s}_1\|_{l-1,\Omega}\nonumber\\
&\leq C\mu\|\boldsymbol\xi^{\rm s}_1\|_{l,\O}<\infty.\label{Psi-s-estimate}
\end{align}
Since $|\mu{\widetilde\bPh}^{\rm s}_1|\leq C\eps^{-\kappa_{1}}$ on $\G_\eps$ and $|\bPs^{\rm s}_1|_{\infty}<\infty$, we obtain
\begin{align}
\|\v^{\rm s}_1\|_{0,q,\G_\eps}&\leq\left(\int_{\G_\eps}|\mu{\widetilde\bPh}^{\rm s}_1|^q dS\right)^{1/q}+\left(\int_{\G_\eps}|\bPs^{\rm s}_1|^q dS\right)^{1/q}\nonumber\\
&\leq\left(\int_{\o_1}^{\o_2}C\eps^{1-\kappa_{1} q}d\th\right)^{1/q}+|\bPs^{\rm s}_1|_\infty\left(\int_{\G_\eps}dS\right)^{1/q}\nonumber\\
&\leq C\eps^{1/q-\kappa_{1}}\o^{1/q}+|\bPs^{\rm s}_1|_{\infty}\,\eps^{1/q}\o^{1/q}\rightarrow 0\qquad\mbox{ as }\,\eps\rightarrow 0^+.\label{Stokes-eps-s-limit}
\end{align}
By \eqref{Stokes-eps-s-limit} and $\left[\w^{\rm s},\sigma^{\rm s}\right]\in{\bf H}^{2-1/p}(\Omega)\times{\rm H}^{1-1/p}(\Omega)$, the inequality \eqref{A-eps-R-estimate1} gives
\begin{equation}\label{Stokes-A-R-s-lim}
\lim_{\eps\rightarrow 0^+}A_{\eps}^{{\rm R}}=0.
\end{equation}
Moreover, since $\mu|\nab\bPh^{\rm s}_i|+|\ph^{\rm s}_i|\leq C\eps^{\kappa_i-1}$ on $\G_\eps$, we have
\begin{equation}\label{Stokes-A-plus-lim}
\begin{aligned}
\left|A_{\eps}^+\right|&\leq C\sum_{i=1}^{M} c^{\rm s}_i\int_{\o_1}^{\o_2}\eps^{\kappa_i-1}|\bPs^{\rm s}_1(\eps\cos\th,\eps\sin\th)|\,\eps \,d\th\\
&\leq C\sum_{i=1}^{M} c^{\rm s}_i \,|\bPs^{\rm s}_1|_{\infty}\,\eps^{\kappa_i}\o\rightarrow 0\qquad\mbox{ as }\,\eps\rightarrow 0^+.
\end{aligned}
\end{equation}
Using \eqref{Stokes-A-R-s-lim} and \eqref{Stokes-A-plus-lim}, the limit of $A_\eps$ given by \eqref{Stokes-A-eps} becomes
\begin{equation}\label{Stokes-A-eps-result}
\lim_{\eps\rightarrow 0^+}A_{\eps}=\lim_{\eps\rightarrow 0^+}A_{\eps}^-.
\end{equation}

We next consider $B_\eps$ in \eqref{f-A-B-equality}.
Since ${\bf u}^{{\rm s}}={\bf g}$ on $\partial\O_\eps\setminus\G_\eps\subset\G$, \eqref{Stokes-decomposition} implies that
\begin{equation}\label{Stokes-B-eps-eq}
\begin{aligned}
B_\eps&=B_{\eps,{\bf g}}+B_{\eps}^{{\rm R}}+B_\eps^+ +B_\eps^-,
\end{aligned}
\end{equation}
where
$$
\begin{aligned}
B_{\eps,{\bf g}}&:=\int_{\partial\O_\eps\setminus\G_\eps}\left(\m{\bf g}\cdot\frac{\partial\v^{\rm s}_1}{\partial\nb}-\left({\bf g}\cdot\nb\right)q^{\rm s}_1\right)dS,\\
B_{\eps}^{{\rm R}}&:=\int_{\G_\eps}\left(\m\w^{{\rm s}}\cdot\frac{\partial\v^{\rm s}_1}{\partial\nb}-\left(\w^{{\rm s}}\cdot\nb\right)q^{\rm s}_1\right)dS,\\
B_{\eps}^+&:=\sum_{i=1}^{M} c^{\rm s}_i\int_{\G_\eps}\left(\m\bPh^{\rm s}_i\cdot\frac{\partial\bPs^{\rm s}_1}{\partial\nb}-\left(\bPh^{\rm s}_i\cdot\nb\right)\ps^{\rm s}_1\right)dS,\\
B_{\eps}^-&:=\sum_{i=1}^{M} c^{\rm s}_i\int_{\G_\eps}\left(\m\bPh^{\rm s}_i\cdot\mu\frac{\partial{\widetilde\bPh}^{\rm s}_1}{\partial\nb}-\left(\bPh^{\rm s}_i\cdot\nb\right)\mu{\widetilde\ph}^{\rm s}_1\right)dS.
\end{aligned}
$$
By \eqref{Stokes-Dual}, a direct calculation gives
\begin{equation}\label{Stokes-Dual-property}
\frac{\pd{\widetilde\bPh}^{\rm s}_1}{\pd\nb}=\mu^{-1}\kappa_1\eps^{-\kappa_1-1}{\widetilde{\bf T}}^{\rm s}_1(\th)\quad\mbox{and}\quad{\widetilde\phi}^{\rm s}_1\nb=-\epsilon^{-\kappa_1-1}{\widetilde\xi}^{\rm s}_1(\th)\,{\bf e}_r(\th)\quad\mbox{ on }\,\G_\eps,
\end{equation}
where ${\bf e}_r(\th):=(\cos\th,\sin\th)^t$.
By \eqref{Stokes-Dual-property} and $\w^{{\rm s}}={\bf g}_J$ for $\theta=\o_2$, we obtain
\begin{align}
&\int_{\Gamma_\epsilon}\left(\m\w^{{\rm s}}\cdot\mu\frac{\partial{\widetilde\bPh}^{\rm s}_1}{\partial{\bf n}}-\left(\w^{{\rm s}}\cdot\nb\right)\mu{\widetilde\phi}^{\rm s}_1\right)dS\nbr\\
&\qquad=\mu\epsilon^{-\kappa_1}\int_{\omega_1}^{\omega_2}\w^{{\rm s}}(\epsilon\cos\theta,\epsilon\sin\theta)\cdot\left(\kappa_1{\widetilde{\bf T}}^{\rm s}_1\left(\th\right)+{\widetilde\xi}^{\rm s}_1(\th)\,{\bf e}_r(\th)\right)d\theta\nbr\\
&\qquad=\mu\epsilon^{-\kappa_1}\bigg({\bf g}_J(\eps\cos\o_2,\eps\sin\o_2)\cdot{\widetilde{\bf A}}^{\rm s}_1(\o_2)\label{Stokes-w-dual-gamma}\\
&\qquad\qquad\qquad\;-\int_{\omega_1}^{\omega_2}\frac{d}{d\theta}\w^{{\rm s}}(\epsilon\cos\theta,\epsilon\sin\theta)\cdot{\widetilde{\bf A}}^{\rm s}_1\left(\th\right)d\theta\bigg),\nbr
\end{align}
where ${\widetilde{\bf A}}^{\rm s}_1(\th):=\int_{\o_1}^\th(\kappa_1{\widetilde{\bf T}}^{\rm s}_1(\sigma)+{\widetilde\xi}^{\rm s}_1(\sigma)\,{\bf e}_r(\sigma))\,d\sigma$.
Since ${\bf g}_J(S_1)={\bf 0}$, H\"{o}lder's inequality and the Sobolev embedding ${\rm H}^{3/2-1/q}(\G_J)\hookrightarrow {\rm W}^{1,q}(\G_J)$ imply that
\begin{align}
\left|{\bf g}_J(\eps\cos\o_2,\eps\sin\o_2)\right|&=\left|\int_{0}^\eps\frac{d}{dr}{\bf g}_J\left(r\cos\o_2,r\sin\o_2\right)dr\right|\nonumber\\
&\leq C\left(\int_0^\eps dr\right)^{1/p}\|{\bf g}_J\|_{1,q,\G_J}\nonumber\\
&\leq C\eps^{1/p}\|{\bf g}_J\|_{3/2-1/q,\G_J},\label{Stokes-gj-estimate}
\end{align}
where $p,q\in(1,\infty)$ are H\"{o}lder conjugates with $\kappa_1 p<1$.
Since ${\widetilde{\bf T}}^{\rm s}_1(\th)$ and ${\widetilde\xi}^{\rm s}_1(\th)$ are bounded for $\th\in\left[\o_1,\o_2\right]$, there exists a constant $C>0$ such that
\begin{equation}\label{Stokes-I-estimate}
|{\widetilde{\bf A}}^{\rm s}_1(\th)|\leq C\qquad\forall\,\th\in\left[\o_1,\o_2\right].
\end{equation}
Using \eqref{Stokes-gj-estimate}, \eqref{Stokes-I-estimate}, $\|r^{-\kappa_1}\|_{0,p,\G_\eps}\leq C\eps^{1/p-\kappa_1}$ and $\|\nab\w^{{\rm s}}\|_{0,q,\G_\eps}\leq C\|\w^{{\rm s}}\|_{2-1/q,\O}$ (see \eqref{A-eps-R-estimate1}), the quantity of \eqref{Stokes-w-dual-gamma} is estimated as follows:
\begin{align}
&\left|\,\int_{\Gamma_\epsilon}\left(\mu\w^{{\rm s}}\cdot\mu\frac{\partial{\widetilde\bPh}^{\rm s}_1}{\partial{\bf n}}-\left(\w^{{\rm s}}\cdot\nb\right)\mu{\widetilde\phi}^{\rm s}_1\right)dS\,\right|\nonumber\\
&\qquad\leq C\mu\eps^{1/p-\kappa_1}\|{\bf g}_J\|_{3/2-1/q,\G_J}+C\mu\int_{\omega_1}^{\omega_2}\epsilon^{-\kappa_1}\left|\nabla\w^{{\rm s}}\left(\epsilon\cos\theta,\epsilon\sin\theta\right)\right|\epsilon \, d\theta\nonumber\\
&\qquad=C\mu\eps^{1/p-\kappa_1}\|{\bf g}_J\|_{3/2-1/q,\G_J}+C\mu\int_{\Gamma_\epsilon}r^{-\kappa_1}\left|\nabla\w^{{\rm s}}\right|dS\nonumber\\
&\qquad\leq C\mu\eps^{1/p-\kappa_1}\|{\bf g}_J\|_{3/2-1/q,\G_J}+C\mu\|r^{-\kappa_1}\|_{0,p,\G_\eps}\|\nab\w^{{\rm s}}\|_{0,q,\G_\eps}\nonumber\\
&\qquad\leq C\mu\eps^{1/p-\kappa_1}\left(\|{\bf g}_J\|_{3/2-1/q,\G_J}+\|\w^{{\rm s}}\|_{2-1/q,\O}\right)\rightarrow 0\qquad\mbox{ as }\,\eps\rightarrow 0^+.\label{Stokes-w-dual-gamma-estimate2}
\end{align}
On the other hand, by $\w^{{\rm s}}\in{\bf H}^2(\Omega)$, we have $|\w^{{\rm s}}|_\infty<\infty$.
By the extended trace theorem and \eqref{Psi-s-estimate}, for $0<\d<\kappa_{1}-1/2$,
\begin{align}
\mu\|\nab\bPs^{\rm s}_1\|_{0,\G_\eps}+\|\ps^{\rm s}_1\|_{0,\G_\eps}&\leq C\left(\mu\|\bPs^{\rm s}_1\|_{1+\d,\partial\O_\eps}+\|\ps^{\rm s}_1\|_{\d,\partial\O_\eps}\right)\nbr\\
&\leq C\left(\mu\|\bPs^{\rm s}_1\|_{3/2+\d,\O_\eps}+\|\ps^{\rm s}_1\|_{1/2+\d,\O_\eps}\right)<\infty.\label{Stokes-ps-gamma-eps}
\end{align}
Hence,
\begin{align}
&\left|\,\int_{\G_\eps}\left(\m\w^{{\rm s}}\cdot\frac{\partial\bPs^{\rm s}_1}{\partial{\bf n}}-\left(\w^{{\rm s}}\cdot\nb\right)\ps^{\rm s}_1\right)dS\,\right|\nbr\\
&\qquad\leq C\|\w^{{\rm s}}\|_{0,\G_\eps}\left(\mu\|\nab\bPs^{\rm s}_1\|_{0,\G_\eps}+\|\ps^{\rm s}_1\|_{0,\G_\eps}\right)\nbr\\
&\qquad\leq C\eps^{1/2}|\w^{{\rm s}}|_{\infty}\left(\mu\|\bPs^{\rm s}_1\|_{3/2+\d,\O_\eps}+\|\ps^{\rm s}_1\|_{1/2+\d,\O_\eps}\right)\rightarrow 0\qquad\mbox{ as }\,\eps\rightarrow 0^+.\label{Stokes-w-dual-gamma-estimate4}
\end{align}
By \eqref{Stokes-w-dual-gamma-estimate2} and \eqref{Stokes-w-dual-gamma-estimate4}, we have
\begin{equation}\label{Stokes-B-w-lim}
\lim_{\eps\rightarrow 0^+}B_{\eps}^{{\rm R}}=0.
\end{equation}
Since $\|\bPh^{\rm s}_i\|_{0,\G_\eps}\leq C\mu^{-1}\eps^{1/2+\kappa_i}$, \eqref{Stokes-ps-gamma-eps} yields that
\begin{equation}\label{Stokes-B-eps-plus-estimate}
|B_\eps^+|\leq C\mu^{-1}\sum_{i=1}^{M}\eps^{1/2+\kappa_i}\left(\mu\|\bPs^{\rm s}_1\|_{3/2+\d,\O_\eps}+\|\ps^{\rm s}_1\|_{1/2+\d,\O_\eps}\right)\rightarrow 0\qquad\mbox{ as }\,\eps\rightarrow 0^+.
\end{equation}
Using the results \eqref{Stokes-B-w-lim} and \eqref{Stokes-B-eps-plus-estimate}, the limit of $B_\eps$ given by \eqref{Stokes-B-eps-eq} becomes
\begin{equation}\label{Stokes-B-eps-result}
\begin{aligned}
\lim_{\eps\rightarrow 0^+}B_{\eps}&=\sum_{j=1}^J\int_{\G_j}\left(\m{\bf g}_j\cdot\frac{\partial\v^{\rm s}_1}{\partial\nb_j}-\left({\bf g}_j\cdot\nb_j\right)q^{\rm s}_1\right)dS+\lim_{\eps\rightarrow 0^+}B_\eps^-.
\end{aligned}
\end{equation}
Combining \eqref{Stokes-A-eps-result} and \eqref{Stokes-B-eps-result} with \eqref{f-A-B-equality}, we have
\begin{equation}\label{Stokes-c1-limit}
\int_\O\left(\f\cdot\v^{\rm s}_1-\zeta q^{\rm s}_1\right)d\x=\,\sum_{j=1}^J\int_{\G_j}\left(\mu{\bf g}_j\cdot\frac{\partial\v^{\rm s}_1}{\partial\nb_j}-\left({\bf g}_j\cdot\nb_j\right)q^{\rm s}_1\right)dS +\lim_{\eps\rightarrow 0^+}(A_\eps^- +B_\eps^-).
\end{equation}
By \eqref{Stokes-Dual-property}, the limit of $A_\eps^-+B_\eps^-$ in \eqref{Stokes-c1-limit} is written as
\begin{align}
&\lim_{\eps\rightarrow 0^+}(A_\eps^- +B_\eps^-)\nbr\\
&\qquad=\lim_{\eps\rightarrow 0}\sum_{i=1}^{M}\epsilon^{\kappa_i-\kappa_1}c^{\rm s}_i\int_{\omega_1}^{\omega_2}\Big((\kappa_i+\kappa_1)\,\bT^{\rm s}_i(\theta)\cdot{\widetilde\bT}^{\rm s}_1(\theta)\nbr\\
&\qquad\qquad\qquad\qquad\qquad\qquad\quad-\xi_i^{\rm s}(\th)\,[{\widetilde\bT}^{\rm s}_1(\th)\cdot\eb_r(\th)]+[\bT^{\rm s}_i(\th)\cdot\eb_r(\th)]\,{\widetilde\xi}^{\rm s}_1(\th)\Big)\,d\th\nbr\\
&\qquad=c^{\rm s}_1\g^{\rm s}_1,\label{Stokes-c1-gamma1-equality}
\end{align}
where $\g^{\rm s}_1\in\mathbb{R}$ is defined by \eqref{Stokes-gamma}.
Applying \eqref{Stokes-c1-gamma1-equality} to \eqref{Stokes-c1-limit}, the desired formula \eqref{Stokes-sif-formula} of $c^{\rm s}_1$ can be derived.

\vspace{0.1cm}

Next, we derive the formula of $c^{\rm s}_2$ in \eqref{Stokes-decomposition}.
Inserting the decomposition \eqref{Stokes-decomposition} to the problem \eqref{Stokes}, we consider the following problem of $[\u^{\rm s}_1,p^{\rm s}_1]:=[\w^{{\rm s}},\sigma^{{\rm s}}]+c^{\rm s}_2[\bPh^{\rm s}_2,\ph^{\rm s}_2]$:
\begin{equation}\label{Stokes-u1}
\left\{\begin{aligned}
-\mu\Delta\u^{\rm s}_1+\nabla p^{\rm s}_1&=\f&&\mbox{ in }\O,\\
\div\,\u^{\rm s}_1&=\zeta&&\mbox{ in }\O,\\
\u^{\rm s}_1&={\bf g}_j&&\mbox{ on }\G_j,\quad j=1,J,\\
\u^{\rm s}_1&={\bf g}_j-c^{\rm s}_1\bPh^{\rm s}_1|_{\G_j}&&\mbox{ on }\G_j,\quad j=2,\ldots,J-1.
\end{aligned}\right.
\end{equation}
Upon integrating by parts, the problem \eqref{Stokes-u1} gives
$$
\begin{aligned}
\int_\O\left(\f\cdot\v^{\rm s}_2-\zeta q^{\rm s}_2\right)d\x&=\lim_{\eps\rightarrow 0^+}\int_{\O_\eps}\left(\left(-\mu\Delta\u^{\rm s}_1+\nab p^{\rm s}_1\right)\cdot\v^{\rm s}_2-\div\,\u^{\rm s}_1 q^{\rm s}_2\right)d\x\\
&=\lim_{\eps\rightarrow 0^+}(C_\eps+D_\eps),
\end{aligned}
$$
where
$$
\begin{aligned}
C_\eps&:=\int_{\partial\O_\eps}\left(-\m\frac{\partial\u^{\rm s}_1}{\partial{\bf n}}\cdot\v^{\rm s}_2+p^{\rm s}_1(\v^{\rm s}_2\cdot{\bf n})\right)dS,\\
D_\eps&:=\int_{\partial\O_\eps}\left(\m\u^{\rm s}_1\cdot\frac{\partial\v^{\rm s}_2}{\partial{\bf n}}-\left(\u^{\rm s}_1\cdot\nb\right)q^{\rm s}_2\right)dS.
\end{aligned}
$$
As in the derivation of \eqref{Stokes-c1-limit}, obtained by \eqref{Stokes-A-eps-result} and \eqref{Stokes-B-eps-result}, we have
\begin{equation}\label{Stokes-C2-eps}
\begin{aligned}
\int_\O\left(\f\cdot\v^{\rm s}_2-\zeta q^{\rm s}_2\right)d\x=&\,\sum_{j=1}^J\int_{\G_j}\left(\mu{\bf g}_j\cdot\frac{\partial\v^{\rm s}_2}{\partial\nb_j}-\left({\bf g}_j\cdot\nb_j\right)q^{\rm s}_2\right)dS\\
&\,-c^{\rm s}_1\sum_{j=2}^{J-1}\int_{\G_j}\left(\mu{\bPh}^{\rm s}_1\cdot\frac{\partial\v^{\rm s}_2}{\partial\nb_j}-\left({\bPh}^{\rm s}_1\cdot\nb_j\right)q^{\rm s}_2\right)dS\\
&\,+\lim_{\eps\rightarrow 0^+}(C_\eps^-+D_\eps^-),
\end{aligned}
\end{equation}
where
$$
\begin{aligned}
C_{\eps}^-&:=c^{\rm s}_2\int_{\G_\eps}\bigg(-\m\frac{\partial\bPh^{\rm s}_2}{\partial{\bf n}}\cdot\mu{\widetilde\bPh}^{\rm s}_2 +\phi^{\rm s}_2\left(\mu{\widetilde\bPh}^{\rm s}_2\cdot\nb\right)\bigg)dS,\\
D_{\eps}^-&:=c^{\rm s}_2\int_{\G_\eps}\bigg(\m\bPh^{\rm s}_2\cdot\mu\frac{\partial{\widetilde\bPh}^{\rm s}_2}{\partial\nb}-\left(\bPh^{\rm s}_2\cdot\nb\right)\mu{\widetilde\ph}^{\rm s}_2\bigg)dS.
\end{aligned}
$$
In a similar manner to \eqref{Stokes-c1-gamma1-equality}, we have $\lim_{\eps\rightarrow 0^+}\left(C_\eps^-+D_\eps^-\right)=c^{\rm s}_2\gamma^{\rm s}_2$, and thus \eqref{Stokes-C2-eps} becomes
$$
\mathcal{C}^{\rm s}_2(\f,\zeta,{\bf g})+c^{\rm s}_1\mathcal{C}^{\rm s}_*=c^{\rm s}_2\g^{\rm s}_2,
$$
which is the desired formula \eqref{Stokes-sif-formula} of $c^{\rm s}_2$.
\end{proof}

\vspace{0.1cm}

We next verify the a priori estimates of the coefficients $c^{\rm s}_i$.

\begin{thm}\label{thm4.3}
Assume that $\f\in{\bf L}^2(\Omega)$, $\zeta\in{\rm H}^1(\Omega)$ and ${\bf g}_j\in{\bf H}^{3/2}(\G_j)$ for $1\leq j\leq J$ with \eqref{g-condition} and ${\bf g}_1(S_1)={\bf g}_J(S_1)={\bf 0}$. Then there exists a constant $C>0$ independent of the viscous number $\m$ such that the coefficients $c^{\rm s}_i$, formulated by \eqref{Stokes-sif-formula}, have the following a priori estimates: for $1\leq i\leq M,$
\begin{equation}\label{sif-stokes-estimate}
|c^{\rm s}_i|\leq C\Bigg(\|\f\|_{0,\Omega}+\mu\|\zeta\|_{1,\Omega}+\mu\sum_{j=1}^J\|{\bf g}_j\|_{3/2,\G_j}\Bigg).
\end{equation}
\end{thm}

\begin{proof}
The quantity $\C^{\rm s}_i(\f,\zeta,{\bf g})$ in \eqref{Stokes-C} is rewritten as
$$
\C^{\rm s}_i(\f,\zeta,{\bf g})=F_i-\sum_{j=1}^J(G_{i,j}^-+G_{i,j}^+),
$$
where
$$
\begin{aligned}
F_i&:=\int_\O\left(\f\cdot\left(\mu{\widetilde\bPh}^{\rm s}_i+\bPs^{\rm s}_i\right)-\zeta\left(\mu{\widetilde\ph}^{\rm s}_i+\ps^{\rm s}_i\right)\right)d\x,\\
G_{i,j}^-&:=\int_{\G_j}\bigg(\m\,{\bf g}_j\cdot\mu\frac{\partial{\widetilde\bPh}^{\rm s}_i}{\partial{\bf n}_j}-\left({\bf g}_j\cdot{\bf n}_j\right)\mu{\widetilde\ph}^{\rm s}_i\bigg)dS,\\
G_{i,j}^+&:=\int_{\G_j}\bigg(\m\,{\bf g}_j\cdot\frac{\partial\bPs^{\rm s}_i}{\partial{\bf n}_j}-\left({\bf g}_j\cdot{\bf n}_j\right)\ps^{\rm s}_i\bigg)dS.
\end{aligned}
$$
Reflecting the form of ${\widetilde\bPh}^{\rm s}_i$ and ${\widetilde\phi}^{\rm s}_i$ given by \eqref{Stokes-Dual} and \cite[Theorem 1.4.4.4]{GE}, we observe that
\begin{equation}\label{Stokes-dual_sing_l2}
\begin{aligned}
\|{\widetilde\bPh}^{\rm s}_i\|_{0,\O}&\leq C\mu^{-1}\left(\int_0^{r_*}r^{-2\kappa_i+1}dr\right)^{1/2}\leq C\mu^{-1},\\
\|{\widetilde\phi}^{\rm s}_i\|_{-1,\O}&=\sup_{0\neq v\in{\rm H}_0^1(\Omega)}\frac{\langle{\widetilde\phi}^{\rm s}_i,v\rangle}{\|v\|_{1,\Omega}}\leq \sup_{0\neq v\in{\rm H}_0^1(\Omega)}\frac{\|r{\widetilde\phi}^{\rm s}_i\|_{0,\Omega}\|r^{-1}v\|_{0,\Omega}}{\|v\|_{1,\Omega}}\\
&\leq C\|r{\widetilde\phi}^{\rm s}_i\|_{0,\Omega}\leq C\left(\int_0^{r_*}r^{-2\kappa_i+1}dr\right)^{1/2}\leq C,
\end{aligned}
\end{equation}
where $r_*:=\mbox{diam}(S_1,\O)$. As seen in \eqref{Psi-s-estimate}, we recall that
\begin{equation}\label{Stokes-psi-estimate}
\mu\|\bPs^{\rm s}_i\|_{l,\O}+\|\ps^{\rm s}_i\|_{l-1,\Omega}\leq C\mu\quad\mbox{ for }\,3/2<l<1+\kappa_1.
\end{equation}
By \eqref{Stokes-dual_sing_l2} and \eqref{Stokes-psi-estimate},
\begin{align}
|F_i|&\leq C\|\f\|_{0,\O}\left(\mu\|{\widetilde\bPh}^{\rm s}_i\|_{0,\O}+\|\bPs^{\rm s}_i\|_{0,\O}\right)+C\|\zeta\|_{1,\O}\left(\mu\|{\widetilde\ph}^{\rm s}_i\|_{-1,\O}+\|\ps^{\rm s}_i\|_{-1,\O}\right)\nbr\\
&\leq C\left(\|\f\|_{0,\O}+\mu\|\zeta\|_{1,\O}\right).\label{Fi-stokes-ineq}
\end{align}
Since $\mu|\pd{\widetilde\bPh}^{\rm s}_i/\pd\nb_j|+|{\widetilde\ph}^{\rm s}_i|\leq Cr^{-\kappa_i-1}$ and ${\bf g}_1(S_1)={\bf 0}$, we obtain
\begin{align}
|G_{i,1}^-|&\leq C\mu\int_0^{r_*}|{\bf g}_1(r\cos\omega_1,r\sin\omega_2)|r^{-\kappa_i-1}\,dr\nbr\\
&=C\mu\int_0^{r_*}\left|\frac{1}{r}\int_0^r \frac{d}{dr}{\bf g}_1(r_1\cos\o_1,r_1\sin\o_1)\,dr_1\right| r^{-\kappa_i}\,dr\nbr\\
&\leq C\mu\left(\int_0^{r_*}\left|\frac{1}{r}\int_0^r\frac{d}{dr}{\bf g}_1(r_1\cos\o_1,r_1\sin\o_1)\,dr_1\right|^p dr\right)^{1/p}\left(\int_0^{r_*}r^{-\kappa_i q}\,dr\right)^{1/q}\nbr\\
&\leq C\mu\left(\int_0^{r_*}\left|\frac{d}{dr}{\bf g}_1(r\cos\o_1,r\sin\o_1)\right|^p dr\right)^{1/p}\qquad\mbox{(by Hardy's inequality)}\nbr\\
&\leq C\mu\|{\bf g}_1\|_{1,p,\G_1}\nbr \leq C\mu\|{\bf g}_1\|_{3/2,\G_1},\label{Gi1-stokes-ineq}
\end{align}
where $p=1/\d$ and $q=(1-\d)^{-1}$ for $0<\d<1-\kappa_i$.
Similarly,
\begin{equation}\label{GiJ-stokes-ineq}
|G_{i,J}^-|\leq C\mu\|{\bf g}_J\|_{3/2,\G_J}.
\end{equation}
Moreover, note that for $j=2,3,\ldots,J-1$,
\begin{align}
|G_{i,j}^-|&=\left|\int_{\G_j}\bigg(\m{\bf g}_j\cdot\mu\frac{\pd{\widetilde\bPh}^{\rm s}_i}{\pd\nb_j}-\big({\bf g}_j\cdot\nb_j\big)\mu{\widetilde\phi}^{\rm s}_i\bigg)dS\right|\nbr\\
&\leq C\mu\int_{\G_j}|{\bf g}_j|r^{-\kappa_i-1}\,dS\nbr \leq C\mu\|{\bf g}_j\|_{0,\G_j}\|r^{-\kappa_i-1}\|_{0,\G_j}\nbr\\
&\leq C\mu\|{\bf g}_j\|_{3/2,\G_j}.\label{Gij-minus-stokes-ineq}
\end{align}
By the Cauchy-Schwarz inequality and \eqref{Stokes-psi-estimate},
\begin{align}
|G_{i,j}^+|&\leq C\|{\bf g}_j\|_{0,\G_j}\left(\m\|\bPs^{\rm s}_i\|_{1,\G_j}+\|\psi^{\rm s}_i\|_{0,\G_j}\right)\nbr\\
&\leq C\|{\bf g}_j\|_{0,\G_j}\left(\m\|\bPs^{\rm s}_i\|_{l,\O}+\|\psi^{\rm s}_i\|_{l-1,\O}\right)\nbr\\
&\leq C\mu\|{\bf g}_j\|_{3/2,\G_j},\label{Gij-plus-stokes-ineq}
\end{align}
where $3/2<l<1+\kappa_1$.
Using \eqref{Fi-stokes-ineq}--\eqref{Gij-plus-stokes-ineq}, we obtain
\begin{align}
\left|\C^{\rm s}_i(\f,\zeta,{\bf g})\right|&\leq |F_i|+\sum_{j=1}^J(|G_{i,j}^-|+|G_{i,j}^+|)\nbr\\
&\leq C\left(\|\f\|_{0,\O}+\mu\|\zeta\|_{1,\O}\right)+C\mu\sum_{j=1}^J\|{\bf g}_j\|_{3/2,\G_j}.\label{Ci-stokes-estimate}
\end{align}

Next, we consider the estimate of $\C^{\rm s}_*$ in \eqref{Stokes-sif-formula}.
Since $|\bPh^{\rm s}_1|\leq C\m^{-1}r^{\kappa_1}$ and $\m|\pd{\widetilde\bPh}^{\rm s}_2/\pd\nb_j|+|{\widetilde\ph}^{\rm s}_2|\leq Cr^{-\kappa_2-1}$, it follows from the extended trace theorem and \eqref{Stokes-psi-estimate} that
\begin{align}
|\C^{\rm s}_*|&=\left|\,\sum_{j=2}^{J-1}\int_{\G_j}\left(\m\bPh^{\rm s}_1\cdot\frac{\pd}{\pd\nb_j}(\mu{\widetilde\bPh}^{\rm s}_2+\bPs^{\rm s}_2)-\left(\bPh^{\rm s}_1\cdot\nb_j\right)(\mu{\widetilde\ph}^{\rm s}_2+\ps^{\rm s}_2)\right)dS \,\right|\nbr\\
&\leq C\sum_{j=2}^{J-1}\int_{\G_j}r^{\kappa_1-\kappa_2-1}dS+C\sum_{j=2}^{J-1}\|\bPh^{\rm s}_1\|_{0,\G_j}\left(\m\|\bPs^{\rm s}_2\|_{1,\G_j}+\|\ps^{\rm s}_2\|_{0,\G_j}\right)\nbr\\
&\leq C+C\m^{-1}\left(\m\|\bPs^{\rm s}_2\|_{l,\O}+\|\ps^{\rm s}_2\|_{l-1,\O}\right) \leq C.\label{Cstar-stokes-estimate}
\end{align}
Since $|\gamma^{\rm s}_i|^{-1}\leq C$, by \eqref{Ci-stokes-estimate}, we have
$$
\begin{aligned}
|c^{\rm s}_1|&=|\g^{\rm s}_1|^{-1}|\C^{\rm s}_1(\f,\zeta,{\bf g})|\\
&\leq C\bigg(\|\f\|_{0,\O}+\mu\|\zeta\|_{1,\O}+\mu\sum_{j=1}^J\|{\bf g}_j\|_{3/2,\G_j}\bigg),
\end{aligned}
$$
and furthermore, using \eqref{Cstar-stokes-estimate},
$$
\begin{aligned}
|c^{\rm s}_2|&\leq|\g^{\rm s}_2|^{-1}(|\C^{\rm s}_2(\f,\zeta,{\bf g})|+|c^{\rm s}_1||\C^{\rm s}_*|)\\
&\leq C\bigg(\|\f\|_{0,\O}+\mu\|\zeta\|_{1,\O}+\mu\sum_{j=1}^J\|{\bf g}_j\|_{3/2,\G_j}\bigg),
\end{aligned}
$$
yielding the desired estimate \eqref{sif-stokes-estimate}.
\end{proof}

\vspace{0.7cm}

\section{Convergence in the incompressible limit of\\the penalized Lam\'{e} system: Theorem \ref{thm1.1}}\label{sec5}

\vspace{0.2cm}

In this section, we finally show the convergence result: Theorem \ref{thm1.1} regarding the penalized system obtained by eliminating the pressure in the following Stokes problem with inhomogeneous boundary condition:
\begin{equation}\label{incomp_stokes}
\left\{\begin{aligned}
-\mu\Delta\u+\nabla p&=\f&&\mbox{ in }\O,\\
\div\,\u&=0&&\mbox{ in }\O,\\
\u&={\bf g}&&\mbox{ on }\G,
\end{aligned}\right.
\end{equation}
where $\f\in{\bf L}^2(\O)$ and ${\bf g}\in{\bf H}^{3/2}(\G)$ are given external functions with $\int_\Gamma\left({\bf g}\cdot{\bf n}\right)dS=0$.

\vspace{0.5cm}

\noindent\underline{\bf Proof of \eqref{sif_eps_convergence} in Theorem \ref{thm1.1}.}

\vspace{0.1cm}

First of all, we consider the following convergence result of the quantity $\g_i^\eps$ given in \eqref{gamma_eps}:
\begin{equation}\label{limit-gamma-eps}
\lim_{\eps\rightarrow 0^+}\g_i^\eps=\mu\g_i^{\rm s},
\end{equation}
where $\g_i^{\rm s}$ is defined by \eqref{Stokes-gamma}.
Set $C_\mu^\eps=1+2\mu\eps$.
By a direct calculation, we obtain
\begin{align}
&2\mu\l_1^\eps\,\bT_1^\eps(\th)\cdot{\widetilde\bT}_1^{\eps}(\th)\label{gamma1-1}\\
&=2\mu\l_1^\eps\Big[\left\{-(C_\mu^\eps-\l_1^\eps)\cos[(1+\l_1^\eps)(\th-\ol\o)]+(C_\mu^\eps\cos(\l_1^\eps\o)-\l_1^\eps\cos\o)\cos[(1-\l_1^\eps)(\th-\ol\o)]\right\}\nbr\\
&\quad\qquad\quad\times\left\{-(C_\mu^\eps+\l_1^\eps)\cos[(1-\l_1^\eps)(\th-\ol\o)]+(C_\mu^\eps\cos(\l_1^\eps\o)+\l_1^\eps\cos\o)\cos[(1+\l_1^\eps)(\th-\ol\o)]\right\}\nbr\\
&\quad\qquad\quad+\left\{-(C_\mu^\eps+\l_1^\eps)\sin[(1+\l_1^\eps)(\th-\ol\o)]+(C_\mu^\eps\cos(\l_1^\eps\o)-\l_1^\eps\cos\o)\sin[(1-\l_1^\eps)(\th-\ol\o)]\right\}\nbr\\
&\quad\qquad\quad\times\left\{-(C_\mu^\eps-\l_1^\eps)\sin[(1-\l_1^\eps)(\th-\ol\o)]+(C_\mu^\eps\cos(\l_1^\eps\o)+\l_1^\eps\cos\o)\sin[(1+\l_1^\eps)(\th-\ol\o)]\right\}\Big]\nbr
\end{align}
and
\begin{align}
\eps^{-1}\left(I_{1}^\eps(\th)+J_{1}^\eps(\th)\right)
&=4\mu\l_1^\eps\Big[\,2C_\mu^\eps\sin[(1-\l_1^\eps)(\th-\ol\o)]\sin[(1+\l_1^\eps)(\th-\ol\o)]\label{gamma1-2}\\
&\qquad\quad\;\;\;-(C_\mu^\eps\cos(\l_1^\eps\o)-\l_1^\eps\cos\o)\sin^2[(1-\l_1^\eps)(\th-\ol\o)]\nbr\\
&\qquad\quad\;\;\;-(C_\mu^\eps\cos(\l_1^\eps\o)+\l_1^\eps\cos\o)\sin^2[(1+\l_1^\eps)(\th-\ol\o)]\,\Big].\nbr
\end{align}
Using \eqref{gamma1-1} and \eqref{gamma1-2}, since $\lim_{\eps\rightarrow 0^+}C_\mu^\eps=1$ and $\lim_{\eps\rightarrow 0^+}\l_1^\eps=\kappa_1$, the limit of $\g_1^\eps$ becomes
$$
\begin{aligned}
&\lim_{\eps\rightarrow 0^+}\g_1^\eps=\int_{\o_1}^{\o_2}\lim_{\eps\rightarrow 0^+}\Big(2\mu\l_1^\eps\,\bT_1^\eps(\th)\cdot{\widetilde\bT}_1^{\eps}(\th)+\eps^{-1}\left(I_{1}^\eps(\th)+J_{1}^\eps(\th)\right)\Big)d\th\\
&=\int_{\o_1}^{\o_2}2\mu\kappa_1\Big[\left\{-(1-\kappa_1)\cos[(1+\kappa_1)(\th-\ol\o)]+(\cos(\kappa_1\o)-\kappa_1\cos\o)\cos[(1-\kappa_1)(\th-\ol\o)]\right\}\\
&\quad\qquad\qquad\quad\times\left\{-(1+\kappa_1)\cos[(1-\kappa_1)(\th-\ol\o)]+(\cos(\kappa_1\o)+\kappa_1\cos\o)\cos[(1+\kappa_1)(\th-\ol\o)]\right\}\\
&\quad\qquad\qquad\quad+\left\{-(1+\kappa_1)\sin[(1+\kappa_1)(\th-\ol\o)]+(\cos(\kappa_1\o)-\kappa_1\cos\o)\sin[(1-\kappa_1)(\th-\ol\o)]\right\}\\
&\quad\qquad\qquad\quad\times\left\{-(1-\kappa_1)\sin[(1-\kappa_1)(\th-\ol\o)]+(\cos(\kappa_1\o)+\kappa_1\cos\o)\sin[(1+\kappa_1)(\th-\ol\o)]\right\}\Big]\\
&\quad\qquad+4\mu\kappa_1\Big[\,2\sin[(1-\kappa_1)(\th-\ol\o)]\sin[(1+\kappa_1)(\th-\ol\o)]\\
&\quad\qquad\qquad\quad\;\;\,-(\cos(\kappa_1\o)-\kappa_1\cos\o)\sin^2[(1-\kappa_1)(\th-\ol\o)]\\
&\quad\qquad\qquad\quad\;\;\,-(\cos(\kappa_1\o)+\kappa_1\cos\o)\sin^2[(1+\kappa_1)(\th-\ol\o)]\,\Big]\,d\th\\
&=\mu\gamma_1^{\rm s}.
\end{aligned}
$$
Similarly,
\begin{align}
&2\mu\l_2^\eps\,\bT_2^\eps(\th)\cdot{\widetilde\bT}_2^{\eps}(\th)\label{gamma2-1}\\
&=2\mu\l_2^\eps\Big[\left\{-(C_\mu^\eps-\l_2^\eps)\cos[(1-\l_2^\eps)(\th-\ol\o)]+(C_\mu^\eps\cos(\l_2^\eps\o)-\l_2^\eps\cos\o)\cos[(1+\l_2^\eps)(\th-\ol\o)]\right\}\nbr\\
&\quad\qquad\;\;\;\,\times\left\{-(C_\mu^\eps+\l_2^\eps)\cos[(1+\l_2^\eps)(\th-\ol\o)]+(C_\mu^\eps\cos(\l_2^\eps\o)+\l_2^\eps\cos\o)\cos[(1-\l_2^\eps)(\th-\ol\o)]\right\}\nbr\\
&\quad\qquad\;\;\;\,+\left\{(C_\mu^\eps-\l_2^\eps)\sin[(1+\l_2^\eps)(\th-\ol\o)]-(C_\mu^\eps\cos(\l_2^\eps\o)+\l_2^\eps\cos\o)\sin[(1-\l_2^\eps)(\th-\ol\o)]\right\}\nbr\\
&\quad\qquad\;\;\;\,\times\left\{(C_\mu^\eps+\l_2^\eps)\sin[(1-\l_2^\eps)(\th-\ol\o)]-(C_\mu^\eps\cos(\l_2^\eps\o)-\l_2^\eps\cos\o)\sin[(1+\l_2^\eps)(\th-\ol\o)]\right\}\Big],\nbr
\end{align}
and
\begin{align}
\eps^{-1}\left(I_{2}^\eps(\th)+J_{2}^\eps(\th)\right)
&=4\mu\l_2^\eps\Big[\,2C_\mu^\eps\cos[(1-\l_2^\eps)(\th-\ol\o)]\cos[(1+\l_2^\eps)(\th-\ol\o)]\label{gamma2-2}\\
&\qquad\quad\;\;\;-(C_\mu^\eps\cos(\l_2^\eps\o)-\l_2^\eps\cos\o)\cos^2[(1+\l_2^\eps)(\th-\ol\o)]\nbr\\
&\qquad\quad\;\;\;-(C_\mu^\eps\cos(\l_2^\eps\o)+\l_2^\eps\cos\o)\cos^2[(1-\l_2^\eps)(\th-\ol\o)]\,\Big].\nbr
\end{align}
By \eqref{gamma2-1} and \eqref{gamma2-2}, since $\lim_{\eps\rightarrow 0^+}C_\mu^\eps=1$ and $\lim_{\eps\rightarrow 0^+}\l_2^\eps=\kappa_2$, we have
$$
\begin{aligned}
&\lim_{\eps\rightarrow 0^+}\g_2^\eps =\int_{\o_1}^{\o_2}\lim_{\eps\rightarrow 0^+}\Big(2\mu\l_2^\eps\,\bT_2^\eps(\th)\cdot{\widetilde\bT}_2^{\eps}(\th)+\eps^{-1}\left(I_{2}^\eps(\th)+J_{2}^\eps(\th)\right)\Big)d\th\\
&=\int_{\o_1}^{\o_2}2\mu\kappa_2\Big[\left\{-(1-\kappa_2)\cos[(1-\kappa_2)(\th-\ol\o)]+(\cos(\kappa_2\o)-\kappa_2\cos\o)\cos[(1+\kappa_2)(\th-\ol\o)]\right\}\\
&\quad\qquad\qquad\quad\times\left\{-(1+\kappa_2)\cos[(1+\kappa_2)(\th-\ol\o)]+(\cos(\kappa_2\o)+\kappa_2\cos\o)\cos[(1-\kappa_2)(\th-\ol\o)]\right\}\\
&\quad\qquad\qquad\quad+\left\{(1-\kappa_2)\sin[(1+\kappa_2)(\th-\ol\o)]-(\cos(\kappa_2\o)+\kappa_2\cos\o)\sin[(1-\kappa_2)(\th-\ol\o)]\right\}\\
&\quad\qquad\qquad\quad\times\left\{(1+\kappa_2)\sin[(1-\kappa_2)(\th-\ol\o)]-(\cos(\kappa_2\o)-\kappa_2\cos\o)\sin[(1+\kappa_2)(\th-\ol\o)]\right\}\Big]\\
&\quad\qquad+4\mu\kappa_2\Big[\,2\cos[(1-\kappa_2)(\th-\ol\o)]\cos[(1+\kappa_2)(\th-\ol\o)]\\
&\quad\qquad\qquad\quad\;\;\,-(\cos(\kappa_2\o)-\kappa_2\cos\o)\cos^2[(1+\kappa_2)(\th-\ol\o)]\\
&\quad\qquad\qquad\quad\;\;\,-(\cos(\kappa_2\o)+\kappa_2\cos\o)\cos^2[(1-\kappa_2)(\th-\ol\o)]\,\Big]\,d\th\\
&=\mu\gamma_2^{\rm s}.
\end{aligned}
$$
So, the convergence result \eqref{limit-gamma-eps} of $\gamma_i^\eps$ for each $i=1,2$ is shown.

\vspace{0.1cm}

Next, we consider the following convergence of the quantity $\C_i^\eps(\f,{\bf g})$ in \eqref{gamma_eps}:
\begin{equation}\label{limit-Ci-eps}
\lim_{\eps\rightarrow 0^+}\C_i^\eps(\f,{\bf g})=\C_i^{\rm s}(\f,0,{\bf g}),
\end{equation}
where $\C_i^{\rm s}(\f,0,{\bf g})$ is defined in \eqref{Stokes-C}.
Actually, we have
$$
\C_i^\eps(\f,{\bf g})-\C_i^{\rm s}(\f,0,{\bf g})=F_i-\sum_{j=1}^J(G_{i,j}^-+G_{i,j}^+),
$$
where
\begin{equation}\label{penalty-FG-definition}
\begin{aligned}
F_i&:=\int_\O\f\cdot\left(\left({\widetilde\bPh}_i^\eps-\mu{\widetilde\bPh}_i^{\rm s}\right)+\left(\bPs_i^\eps-\bPs_i^{\rm s}\right)\right)d\x,\\
G_{i,j}^-&:=\int_{\G_j}\left(\m\,{\bf g}_j\cdot\frac{\partial}{\partial{\bf n}_j}\left({\widetilde\bPh}_i^\eps-\mu{\widetilde\bPh}_i^{\rm s}\right)+\left({\bf g}_j\cdot{\bf n}_j\right)\left(\eps^{-1}\div\,{\widetilde\bPh}_i^\eps+\mu{\widetilde\phi}_i^{\rm s}\right)\right)dS,\\
G_{i,j}^+&:=\int_{\G_j}\left(\m\,{\bf g}_j\cdot\frac{\partial}{\partial{\bf n}_j}\left({\bPs}_i^\eps-{\bPs}_i^{\rm s}\right)+\left({\bf g}_j\cdot{\bf n}_j\right)\left(\eps^{-1}\div\,{\bPs}_i^\eps+{\psi}_i^{\rm s}\right)\right)dS.
\end{aligned}
\end{equation}
Applying the penalty method to \eqref{Stokes-ps}, we see that for $\eps>0$,
\begin{equation}\label{penalty-ps}
\left\{\begin{aligned}
 -\mu\Delta{\widetilde\bPs}^\eps_i+\nabla{\widetilde\psi}^\eps_i&={\bf 0}&&\mbox{ in }\Omega,\\
 \div\,{\widetilde\bPs}^\eps_i&=-\eps{\widetilde\psi}^\eps_i&&\mbox{ in }\Omega,\\
 {\widetilde\bPs}^\eps_i&={\bf 0}&&\mbox{ on }\Gamma_j,\quad j=1,J,\\
 {\widetilde\bPs}^\eps_i&=-\mu{\widetilde\bPh}^{\rm s}_i&&\mbox{ on }\Gamma_j,\quad j=2,\ldots,J-1.
\end{aligned}\right.
\end{equation}
Subtracting \eqref{penalty-ps} from \eqref{Stokes-ps} and letting ${\boldsymbol\E}_1:=\bPs_i^{\rm s}-{\widetilde\bPs}^\eps_i$ and $e_1:=\psi_i^{\rm s}-{\widetilde\psi}^\eps_i$, we have
\begin{equation}\label{penalty-ps-error}
\left\{\begin{aligned}
 -\mu\Delta{\boldsymbol\E}_1+\nabla e_1&={\bf 0}&&\mbox{ in }\Omega,\\
 \div\,{\boldsymbol\E}_1&=\eps{\widetilde\psi}^\eps_i&&\mbox{ in }\Omega,\\
 {\boldsymbol\E}_1&={\bf 0}&&\mbox{ on }\Gamma.
\end{aligned}\right.
\end{equation}
Let $3/2<l<1+\min\{\l_1^\eps,\kappa_1\}$. By \cite[Theorem 2.1]{CKT}, the difference pair $\left[{\boldsymbol\E}_1,e_1\right]\in{\bf H}_0^l(\Omega)\times\bar{\rm H}^{l-1}(\Omega)$ of \eqref{penalty-ps-error} satisfies
$$
\begin{aligned}
\mu\|{\boldsymbol\E}_1\|_{l,\O}+\|e_1\|_{l-1,\O}&\leq C\mu\epsilon\|{\widetilde\psi}^\eps_i\|_{l-1,\O}\\
&\leq C\mu\epsilon\|e_1\|_{l-1,\O}+C\mu\eps\|\psi_i^{\rm s}\|_{l-1,\O}.
\end{aligned}
$$
Assuming that the penalty parameter $\eps>0$ is sufficiently small with $C\mu\eps<1/2$, by \eqref{Stokes-psi-estimate}, we have
\begin{equation}\label{E1-convergence-estimate}
\mu\|{\boldsymbol\E}_1\|_{l,\O}+2^{-1}\|e_1\|_{l-1,\O}\leq C\mu\eps\|\psi_i^{\rm s}\|_{l-1,\O}\rightarrow 0\qquad\mbox{ as }\,\eps\rightarrow 0^+.
\end{equation}
Also, the system \eqref{penalty-ps} gives the following equivalent elliptic problem:
\begin{equation}\label{lame-eps-ps}
\left\{\begin{aligned}
 -\mu\Delta{\widetilde\bPs}^\eps_i-\eps^{-1}\nabla\div\,{\widetilde\bPs}^\eps_i&={\bf 0}&&\mbox{ in }\Omega,\\
 {\widetilde\bPs}^\eps_i&={\bf 0}&&\mbox{ on }\Gamma_j,\quad j=1,J,\\
 {\widetilde\bPs}^\eps_i&=-\mu{\widetilde\bPh}^{\rm s}_i&&\mbox{ on }\Gamma_j,\quad j=2,\ldots,J-1.
\end{aligned}\right.
\end{equation}
Subtracting \eqref{lame-eps-ps} from \eqref{Psi-eps-prob} and letting ${\boldsymbol\E}_2:=\bPs_i^\eps-{\widetilde\bPs}_i^\eps$, we have
\begin{equation}\label{lame-eps-ps-error}
\left\{\begin{aligned}
 -\mu\Delta{\boldsymbol\E}_2-\eps^{-1}\nabla\div\,{\boldsymbol\E}_2&={\bf 0}&&\mbox{ in }\Omega,\\
 {\boldsymbol\E}_2&={\bf 0}&&\mbox{ on }\Gamma_j,\quad j=1,J,\\
 {\boldsymbol\E}_2&=-{\widetilde\bPh}^\eps_i+\mu{\widetilde\bPh}^{\rm s}_i&&\mbox{ on }\Gamma_j,\quad j=2,\ldots,J-1.
\end{aligned}\right.
\end{equation}
By Theorem \ref{thm2.2} and $\lim_{\eps\rightarrow 0}\|{\widetilde\bPh}^\eps_i-\mu{\widetilde\bPh}^{\rm s}_i\|_{l-1/2,\G_j}=0$ for $j=2,\ldots J-1$, the difference ${\boldsymbol\E}_2$ of \eqref{lame-eps-ps-error} satisfies
\begin{equation}\label{E2-convergence-estimate}
\mu\|{\boldsymbol\E}_2\|_{l,\O}+\eps^{-1}\|\div\,{\boldsymbol\E}_2\|_{l-1,\O}\leq C\mu\sum_{j=2}^{J-1}\left\|{\widetilde\bPh}^\eps_i-\mu{\widetilde\bPh}^{\rm s}_i\right\|_{l-1/2,\G_j}\rightarrow 0\qquad\mbox{ as }\,\eps\rightarrow 0^+.
\end{equation}
Combining \eqref{E1-convergence-estimate} with \eqref{E2-convergence-estimate}, since ${\widetilde\psi}_i^\eps=-\eps^{-1}\div\,{\widetilde\bPs}_i^\eps$ in $\O$, we have
\begin{align}
&\mu\|\bPs_i^\eps-\bPs_i^{\rm s}\|_{l,\O}+\|\eps^{-1}\div\,\bPs_i^\eps+\psi_i^{\rm s}\|_{l-1,\O}\nbr\\
&\qquad\leq \mu\|{\boldsymbol\E}_1\|_{l,\O}+\mu\|{\boldsymbol\E}_2\|_{l,\O}+\|e_1\|_{l-1,\O}+\eps^{-1}\|\div\,{\boldsymbol\E}_2\|_{l-1,\O}\rightarrow 0\qquad\mbox{ as }\,\eps\rightarrow 0^+.\label{E-combined-convergence}
\end{align}
Since $\lim_{\eps\rightarrow 0^+}\|{\widetilde\bPh}_i^\eps-\mu{\widetilde\bPh}_i^{\rm s}\|_{0,\O}=0$, by \eqref{E-combined-convergence}, we obtain
\begin{equation}\label{Fi-convergence-result}
|F_i|\leq C\|\f\|_{0,\O}\left(\|{\widetilde\bPh}_i^\eps-\mu{\widetilde\bPh}_i^{\rm s}\|_{0,\O}+\|\bPs_i^\eps-\bPs_i^{\rm s}\|_{0,\O}\right)\rightarrow 0\qquad\mbox{ as }\,\eps\rightarrow 0^+.
\end{equation}
Let $m_i:=\max\{\lambda_i^\eps,\kappa_i\}$.
A direct calculation gives that
\begin{align}
\nabla({\widetilde\bPh}_i^\eps-\mu{\widetilde\bPh}_i^{\rm s})&=r^{-\lambda_i^\eps-1}{\bf M}_i^\eps(\th)-r^{-\kappa_i-1}{\bf M}_i^{\rm s}(\th)\nbr\\
&=r^{-m_i-1}{\widetilde{\bf M}}^\eps_i(r,\th),\label{penalty-note1}
\end{align}
where
$$
\begin{aligned}
{\bf M}_i^\eps(\th)&:=-\lambda_i^\eps\eb_r(\th)({\widetilde\bT}_i^\eps(\th))^t+\eb_\th(\th)(\pd_\th{\widetilde\bT}_i^\eps(\th))^t,\\
{\bf M}_i^{\rm s}(\th)&:=-\kappa_i\eb_r(\th)({\widetilde\bT}_i^{\rm s}(\th))^t+\eb_\th(\th)(\pd_\th{\widetilde\bT}_i^{\rm s}(\th))^t,\\
{\widetilde{\bf M}}^\eps_i(r,\th)&:=r^{m_i-\lambda_i^\eps}\left({\bf M}_i^\eps(\th)-{\bf M}_i^{\rm s}(\th)\right)+\big(r^{m_i-\lambda_i^\eps}-r^{m_i-\kappa_i}\big){\bf M}_i^{\rm s}(\th).
\end{aligned}
$$
Also, by a direct calculation, we have
\begin{equation}\label{eps_div_dual}
\eps^{-1}\div{\widetilde\bPh}_i^\eps=\left\{\begin{aligned}
&4\mu\lambda_1^\eps r^{-\lambda_1^\eps-1}\sin[(1+\lambda_1^\eps)(\th-\ol\o)]&&\qquad\mbox{ for }\,i=1,\\
&4\mu\lambda_2^\eps r^{-\lambda_2^\eps-1}\cos[(1+\lambda_2^\eps)(\th-\ol\o)]&&\qquad\mbox{ for }\,i=2,
\end{aligned}\right.
\end{equation}
yielding
\begin{align}
\eps^{-1}\div{\widetilde\bPh}_i^\eps+\mu{\widetilde\phi}_i^{\rm s}&=4\mu\left(r^{-\l_i^\eps-1}N_i^\eps(\th)-r^{-\kappa_i-1}N_i^{\rm s}(\th)\right)\nbr\\
&=4\mu r^{-m_i-1}{\widetilde N}_i^\eps(r,\th),\label{penalty-note2}
\end{align}
where
$$
\begin{aligned}
N_i^\eps(\th)&:=\left\{\begin{aligned}
&\l_i^\eps\sin[(1+\l_i^\eps)(\th-\ol\omega)]&&\qquad\mbox{ for }\,i=1,\\
&\l_i^\eps\cos[(1+\l_i^\eps)(\th-\ol\omega)]&&\qquad\mbox{ for }\,i=2,
\end{aligned}\right.\\
N_i^{\rm s}(\th)&:=\left\{\begin{aligned}
&\kappa_i\sin[(1+\kappa_i)(\th-\ol\omega)]&&\qquad\mbox{ for }\,i=1,\\
&\kappa_i\cos[(1+\kappa_i)(\th-\ol\omega)]&&\qquad\mbox{ for }\,i=2,
\end{aligned}\right.\\
{\widetilde N}_i^\eps(r,\th)&:= r^{m_i-\lambda_i^\eps}\left(N_i^\eps(\th)-N_i^{\rm s}(\th)\right)+\left(r^{m_i-\lambda_i^\eps}-r^{m_i-\kappa_i}\right)N_i^{\rm s}(\th).
\end{aligned}
$$
Using \eqref{penalty-note1} and \eqref{penalty-note2}, since ${\bf g}_1(S_1)={\bf 0}$ and $Q_i^\eps:=|{\widetilde{\bf M}}_i^\eps|_\infty+|{\widetilde N}_i^\eps|_\infty\rightarrow 0$ as $\eps\rightarrow 0^+$, we obtain
\begin{align}
|G_{i,1}^-|&\leq C\mu Q_i^\eps\int_0^{r_*}\left|{\bf g}_1(r\cos\omega_1,r\sin\omega_2)\right|r^{-m_i-1}\,dr\nbr\\
&=C \mu Q_i^\eps\int_0^{r_*}\left|\frac{1}{r}\int_0^r \frac{d}{dr}{\bf g}_1(r_1\cos\o_1,r_1\sin\o_1)\,dr_1\right| r^{-m_i}\,dr\nbr\\
&\leq C \mu Q_i^\eps\left(\int_0^{r_*}\left|\frac{1}{r}\int_0^r\frac{d}{dr}{\bf g}_1(r_1\cos\o_1,r_1\sin\o_1)\,dr_1\right|^p dr\right)^{1/p}\left(\int_0^{r_*}r^{-m_i q}\,dr\right)^{1/q}\nbr\\
&\leq C \mu Q_i^\eps\left(\int_0^{r_*}\left|\frac{d}{dr}{\bf g}_1(r\cos\o_1,r\sin\o_1)\right|^p dr\right)^{1/p}\qquad\mbox{(by Hardy's inequality)}\nbr\\
&\leq C\mu Q_i^\eps\|{\bf g}_1\|_{1,p,\G_1} \leq C\mu Q_i^\eps\|{\bf g}_1\|_{3/2,\G_1}\rightarrow 0\qquad\mbox{ as }\,\eps\rightarrow 0^+,\label{Gi1-convergence-result}
\end{align}
where $p=1/\d$ and $q=(1-\d)^{-1}$ for $0<\d<1-m_i$.
Similarly, 
\begin{equation}\label{GiJ-convergence-result}
|G_{i,J}^-|\leq C\mu Q_i^\eps\|{\bf g}_J\|_{3/2,\G_J}\rightarrow 0\qquad\mbox{ as }\,\eps\rightarrow 0^+.
\end{equation}
Furthermore, for $j=2,3,\ldots,J-1$,
\begin{align}
|G_{i,j}^-|&\leq C\mu Q_i^\eps\int_{\G_j}|{\bf g}_j|r^{-m_i-1}\,dS\nbr \leq C\mu Q_i^\eps\|{\bf g}_j\|_{0,\G_j}\|r^{-m_i-1}\|_{0,\G_j}\nbr\\
&\leq C\mu Q_i^\eps\|{\bf g}_j\|_{3/2,\G_j}\rightarrow 0\qquad\mbox{ as }\,\eps\rightarrow 0^+.\label{Gij-convergence-result}
\end{align}
By the Cauchy-Schwarz inequality and \eqref{E-combined-convergence}, 
\begin{align}
|G_{i,j}^+|&\leq C\|{\bf g}_j\|_{0,\G_j}\left(\m\|\bPs_i^\eps-\bPs^{\rm s}_i\|_{1,\G_j}+\|\eps^{-1}\div\bPs_i^\eps+\psi^{\rm s}_i\|_{0,\G_j}\right)\nbr\\
&\leq C\|{\bf g}_j\|_{3/2,\G_j}\left(\m\|\bPs_i^\eps-\bPs^{\rm s}_i\|_{l,\O}+\|\eps^{-1}\div\bPs_i^\eps+\psi^{\rm s}_i\|_{l-1,\O}\right)\rightarrow 0\qquad\mbox{ as }\,\eps\rightarrow 0^+,\label{Gij-plus-convergence-result}
\end{align}
where $3/2<l<1+\min\{\l_1^\eps,\kappa_1\}$. Using \eqref{Fi-convergence-result} and \eqref{Gi1-convergence-result}--\eqref{Gij-plus-convergence-result}, we obtain
$$
\begin{aligned}
\left|\C_i^\eps(\f,{\bf g})-\C_i^{\rm s}(\f,0,{\bf g})\right|&\leq |F_i|+\sum_{j=1}^J(|G_{i,j}^-|+|G_{i,j}^+|)\rightarrow 0\qquad\mbox{ as }\,\eps\rightarrow 0^+,
\end{aligned}
$$
and then the convergence result \eqref{limit-Ci-eps} is shown.

Moreover, we consider the following convergence of $\C_*^\eps$ in \eqref{gamma_eps}:
\begin{equation}\label{limit-Cstar-eps}
\lim_{\eps\rightarrow 0^+}\C_*^\eps=\mu\C_*^{\rm s},
\end{equation}
where $\C_*^{\rm s}$ is defined in \eqref{Stokes-C}.
Note that
$$
\begin{aligned}
\C_*^\eps-\mu\C_*^{\rm s}&=\sum_{j=2}^{J-1}\left(K_j^1+K_j^2+K_j^3+K_j^4\right),
\end{aligned}
$$
where
$$
\begin{aligned}
K_j^1&:=\int_{\G_j}\left(\mu\bPh_1^\eps\cdot\frac{\partial}{\partial{\bf n}_j}\left({\widetilde\bPh}_2^\eps-\mu{\widetilde\bPh}_2^{\rm s}\right)+\left(\bPh_1^\eps\cdot{\bf n}_j\right)\left(\eps^{-1}\div{\widetilde\bPh}_2^\eps+\mu{\widetilde\phi}_2^{\rm s}\right)\right)dS,\\
K_j^2&:=\int_{\G_j}\left(\mu\bPh_1^\eps\cdot\frac{\partial}{\partial{\bf n}_j}\left(\bPs_2^\eps-\bPs_2^{\rm s}\right)+\left(\bPh_1^\eps\cdot{\bf n}_j\right)\left(\eps^{-1}\div\bPs_2^\eps+\ps_2^{\rm s}\right)\right)dS,\\
K_j^3&:=\int_{\G_j}\left(\mu\left(\bPh_1^\eps-\mu\bPh_1^{\rm s}\right)\cdot\mu\frac{\partial}{\partial{\bf n}_j}{\widetilde\bPh}_2^{\rm s}-\left(\left(\bPh_1^\eps-\mu\bPh_1^{\rm s}\right)\cdot{\bf n}_j\right)\mu{\widetilde\phi}_2^{\rm s}\right)dS,\\
K_j^4&:=\int_{\G_j}\left(\mu\left(\bPh_1^\eps-\mu\bPh_1^{\rm s}\right)\cdot\frac{\partial}{\partial{\bf n}_j}\bPs_2^{\rm s}-\left(\left(\bPh_1^\eps-\mu\bPh_1^{\rm s}\right)\cdot{\bf n}_j\right)\ps_2^{\rm s}\right)dS.
\end{aligned}
$$
By \eqref{penalty-note1} and \eqref{penalty-note2}, since $Q_i^\eps:=|{\widetilde{\bf M}}_i^\eps|_\infty+|{\widetilde N}_i^\eps|_\infty\rightarrow 0$ as $\eps\rightarrow 0^+$, for $j=2,\ldots,J-1$,
\begin{align}
|K_j^1|&\leq C\mu Q_i^\eps\int_{\G_j}r^{\l_1^\eps-m_2-1}\,dS\nbr\\
&\leq C\mu Q_i^\eps (r_{min})^{\l_1^\eps-m_2-1}\rightarrow 0\qquad\mbox{ as }\,\eps\rightarrow 0^+,\label{K1-convergence-result}
\end{align}
where $r_{min}>0$ is the number defined by $r_{min}=\inf\{|{\bf x}|:{\bf x}\in\G_j\mbox{ for }j=2,\ldots,J-1\}$.
Also, since $|\bPh_1^\eps|\leq Cr^{\l_1^\eps}$, by \eqref{E-combined-convergence}, we obtain that for $j=2,\ldots,J-1$,
\begin{align}
|K_j^2|&\leq C\int_{\G_j}r^{\l_1^\eps}\left(\mu|\nabla(\bPs_2^\eps-\bPs_2^{\rm s})|+|\eps^{-1}\div\bPs_2^\eps+\psi_2^{\rm s}|\right)dS\nbr\\
&\leq C\|r^{\l_1^\eps}\|_{0,\G_j}\left(\mu\|\bPs_2^\eps-\bPs_2^{\rm s}\|_{1,\G_j}+\|\eps^{-1}\div\bPs_2^\eps+\psi_2^{\rm s}\|_{0,\G_j}\right)\nbr\\
&\leq C\left(\mu\|\bPs_2^\eps-\bPs_2^{\rm s}\|_{l,\O}+\|\eps^{-1}\div\bPs_2^\eps+\psi_2^{\rm s}\|_{l-1,\O}\right)\rightarrow 0\qquad\mbox{ as }\,\eps\rightarrow 0^+.\label{K2-convergence-result}
\end{align}
Since $\mu|\partial{\widetilde\bPh}_2^{\rm s}/\partial{\bf n}_j|+|{\widetilde\phi}_2^{\rm s}|\leq Cr^{-\kappa_2-1}$ and $\lim_{\eps\rightarrow 0^+}\|\bPh_1^\eps-\mu\bPh_1^{\rm s}\|_{0,\G_j}=0$, for $j=2,\ldots,J-1$,
\begin{align}
|K_j^3|&\leq C\mu\int_{\G_j}|\bPh_1^\eps-\mu\bPh_1^{\rm s}|r^{-\kappa_2-1}dS\nbr\\
&\leq C\mu(r_{min})^{-\kappa_2-1}\|\bPh_1^\eps-\mu\bPh_1^{\rm s}\|_{0,\G_j}\rightarrow 0\qquad\mbox{ as }\,\eps\rightarrow 0^+,\label{K3-convergence-result}
\end{align}
and by \eqref{Stokes-psi-estimate},
\begin{align}
|K_j^4|&\leq C\|\bPh_1^\eps-\mu\bPh_1^{\rm s}\|_{0,\G_j}\left(\mu\|\bPs_2^{\rm s}\|_{1,\G_j}+\|\psi_2^{\rm s}\|_{0,\G_j}\right)\nbr\\
&\leq C\|\bPh_1^\eps-\mu\bPh_1^{\rm s}\|_{0,\G_j}\left(\mu\|\bPs_2^{\rm s}\|_{l,\O}+\|\psi_2^{\rm s}\|_{l-1,\O}\right)\rightarrow 0\qquad\mbox{ as }\,\eps\rightarrow 0^+.\label{K4-convergence-result}
\end{align}
Appealing to \eqref{K1-convergence-result}--\eqref{K4-convergence-result}, the proof of \eqref{limit-Cstar-eps} is complete.

In conclusion, the convergence results \eqref{limit-gamma-eps} and \eqref{limit-Ci-eps} give
$$
\begin{aligned}
&c_1^\eps-\m^{-1}c_1^{\rm s} =\frac{\C_1^\eps(\f,{\bf g})}{\g_1^\eps}-\frac{\C_1^{\rm s}(\f,0,{\bf g})}{\m\g_1^{\rm s}}=\frac{\m\g_1^{\rm s}\C_1^\eps(\f,{\bf g})-\g_1^\eps\C_1^{\rm s}(\f,0,{\bf g})}{\mu\g_1^\eps\g_1^{\rm s}}\\
&\qquad=\frac{1}{\m\g_1^\eps\g_1^{\rm s}}\left[\left(\m\g_1^{\rm s}-\g_1^\eps\right)\C_1^\eps(\f,{\bf g})+\g_1^\eps\left(\C_1^\eps(\f,{\bf g})-\C_1^{\rm s}(\f,0,{\bf g})\right)\right]\rightarrow 0\qquad\mbox{ as }\,\eps\rightarrow 0^+,
\end{aligned}
$$
and furthermore, by \eqref{limit-Cstar-eps},
$$
\begin{aligned}
&c_2^\eps-\m^{-1}c_2^{\rm s} =\frac{\C_2^\eps(\f,{\bf g})+c_1^\eps\C_*^\eps}{\g_2^\eps}-\frac{\C_2^{\rm s}(\f,0,{\bf g})+ c_1^{\rm s}\C_*^{\rm s}}{\m\g_2^{\rm s}}\\
&\qquad=\frac{\m\g_2^{\rm s}\C_2^\eps(\f,{\bf g})+\m\g_2^{\rm s}c_1^\eps\C_*^\eps-\g_2^\eps\C_2^{\rm s}(\f,0,{\bf g})-\g_2^\eps c_1^{\rm s}\C_*^{\rm s}}{\m\g_2^\eps\g_2^{\rm s}}\\
&\qquad=\frac{1}{\m\g_2^\eps\g_2^{\rm s}}\left[\left(\m\g_2^{\rm s}-\g_2^\eps\right)\C_2^\eps(\f,{\bf g})+\g_2^\eps\left(\C_2^\eps(\f,{\bf g})-\C_2^{\rm s}(\f,0,{\bf g})\right)\right]\\
&\qquad\quad\,+\frac{1}{\m\g_2^\eps\g_2^{\rm s}}\left[\left(\m\g_2^{\rm s}-\g_2^\eps\right)c_1^\eps\C_*^\eps+\g_2^\eps\left(c_1^\eps-\m^{-1}c_1^{\rm s}\right)\C_*^\eps+\g_2^\eps \mu^{-1} c_1^{\rm s}\left(\C_*^\eps-\mu\C_*^{\rm s}\right)\right]\\
&\qquad\rightarrow 0\qquad\mbox{ as }\,\eps\rightarrow 0^+.
\end{aligned}
$$
So, the proof of \eqref{sif_eps_convergence} in Theorem \ref{thm1.1} is done.
\qed

\vspace{0.5cm}

\noindent\underline{\bf Proof of \eqref{regular_eps_convergence} in Theorem \ref{thm1.1}.}

\vspace{0.1cm}

In this proof, it is assumed that $M=2$.
With $c_1^\eps$ and $c_2^\eps$ obtained by \eqref{coeff_eps_formula}, the penalized system \eqref{Lame_eps} and its related corner singularity expansion \eqref{decomposition_eps} give the regular part $\w^\eps$ in \eqref{decomposition_eps} that satisfies
\begin{equation}\label{regular-eps}
\left\{\begin{aligned}
-\mu\Delta\w^\eps-\frac{1}{\eps}\nabla\div\,\w^\eps&=\f&&\mbox{ in }\O,\\
\w^\eps&={\bf g}-\sum_{i=1}^2 c_i^\eps \,\bPh_i^\eps&&\mbox{ on }\G.
\end{aligned}\right.
\end{equation}
Owing to \eqref{Stokes-decomposition} with $M=2$, the Stokes problem \eqref{incomp_stokes} becomes
\begin{equation}\label{regular_incomp_stokes}
\left\{\begin{aligned}
-\mu\Delta\w^{\rm s}+\nabla\sigma^{\rm s}&=\f&&\mbox{ in }\O,\\
\div\,\w^{\rm s}&=0&&\mbox{ in }\O,\\
\w^{\rm s}&={\bf g}-\sum_{i=1}^{2}c_i^{\rm s}\,\bPh_i^{\rm s}&&\mbox{ on }\Gamma.
\end{aligned}\right.
\end{equation}
Applying the penalty method to \eqref{regular_incomp_stokes}, we have the following elliptic problem as seen in \eqref{Lame_eps}:
\begin{equation}\label{tilde-regular-eps}
\left\{\begin{aligned}
-\mu\Delta{\widetilde\w}^\eps-\frac{1}{\eps}\nabla\div\,{\widetilde\w}^\eps&=\f&&\mbox{ in }\O,\\
{\widetilde\w}^\eps&={\bf g}-\sum_{i=1}^2 c_i^{\rm s}\,\bPh_i^{\rm s}&&\mbox{ on }\G.
\end{aligned}\right.
\end{equation}
Letting ${\widetilde\sigma}^\eps=-\eps^{-1}\div\,{\widetilde\w}^\eps$, by Theorem 5.3 in \cite[Chapter 1]{GRF}, 
\begin{equation}\label{w-tilde-eps-ineq}
\mu\|{\widetilde\w}^\eps-\w^{\rm s}\|_{1,\O}+\|{\widetilde\sigma}^\eps-\sigma^{\rm s}\|_{0,\O}\leq C\eps\left(\|\f\|_{-1,\O}+\mu\|{\bf g}\|_{1/2,\G}+\mu\sum_{i=1}^2\left|c_i^{\rm s}\right|\left\|\bPh_i^{\rm s}\right\|_{1/2,\G}\right).
\end{equation}
From Theorem \ref{thm4.3}, note that $\left|c_i^{\rm s}\right|<\infty$, and by the trace theorem, $\|\bPh_i^{\rm s}\|_{1/2,\G}\leq C\|\bPh_i^{\rm s}\|_{1,\O}<\infty$. Hence, the inequality \eqref{w-tilde-eps-ineq} becomes
\begin{equation}\label{w-tilde-eps-ineq-final}
\mu\|{\widetilde\w}^\eps-\w^{\rm s}\|_{1,\O}+\|{\widetilde\sigma}^\eps-\sigma^{\rm s}\|_{0,\O}\rightarrow 0\qquad\mbox{ as }\,\eps\rightarrow 0^+.
\end{equation}

Next, we consider the convergence of the differences ${\boldsymbol\E}:=\w^\eps-{\widetilde\w}^\eps$ and $e:=\sigma^\eps-{\widetilde\sigma}^\eps$.
Subtracting \eqref{tilde-regular-eps} from \eqref{regular-eps}, 
\begin{equation}\label{E-eq}
\left\{\begin{aligned}
-\mu\Delta{\boldsymbol\E}-\frac{1}{\eps}\nabla\div\,{\boldsymbol\E}&={\bf 0}&&\mbox{ in }\O,\\
{\boldsymbol\E}&=\sum_{i=1}^2 \left(c_i^{\rm s}\,\bPh_i^{\rm s}-c_i^\eps\,\bPh_i^\eps\right)&&\mbox{ on }\G.
\end{aligned}\right.
\end{equation}
Let $1<l<1+\min\{\lambda_1^\eps,\kappa_1\}$.
By Theorem \ref{thm2.2}, the difference ${\boldsymbol\E}$ of \eqref{E-eq} satisfies
\begin{align}
&\mu\|{\boldsymbol\E}\|_{l,\Omega}+\eps^{-1}\|\div\,{\boldsymbol\E}\|_{l-1,\Omega}\nbr\\
&\qquad\leq C\m\left\|\,\sum_{i=1}^2\left(c_i^{\rm s}\,\bPh_i^{\rm s}-c_i^\eps\,\bPh_i^\eps\right)\right\|_{l-1/2,\Gamma}\nbr\\
&\qquad\leq C\mu\sum_{i=1}^2 |\mu^{-1}c_i^{\rm s}-c_i^\eps|\|\bPh_i^\eps\|_{l-1/2,\Gamma}+C\sum_{i=1}^2 |c_i^{\rm s}|\|\mu\bPh_i^{\rm s}-\bPh_i^\eps\|_{l-1/2,\Gamma}.\label{E-estimate1}
\end{align}
Here, the convergence \eqref{sif_eps_convergence} implies that $|\mu^{-1}c_i^{\rm s}-c_i^\eps|\rightarrow 0$ as $\eps\rightarrow 0^+$, and also Theorem \ref{thm4.3} gives that $|c_i^{\rm s}|<\infty$.
By the extended trace theorem, $\|\bPh_i^\eps\|_{l-1/2,\G}\leq C\|\bPh_i^\eps\|_{l,\O}<\infty$, and by a direct calculation, 
$$
\lim_{\eps\rightarrow 0^+}\|\mu\bPh_i^{\rm s}-\bPh_i^\eps\|_{l-1/2,\G}=0.
$$
Hence, with $e=-\eps^{-1}\div\,{\boldsymbol\E}$, the inequality \eqref{E-estimate1} becomes
\begin{equation}\label{E-ineq-final}
\mu\|{\boldsymbol\E}\|_{l,\Omega}+\|e\|_{l-1,\Omega}=\mu\|{\boldsymbol\E}\|_{l,\Omega}+\eps^{-1}\|\div\,{\boldsymbol\E}\|_{l-1,\Omega}\rightarrow 0\qquad\mbox{ as }\,\eps\rightarrow 0^+.
\end{equation}
Combining \eqref{w-tilde-eps-ineq-final} with \eqref{E-ineq-final}, the desired convergence result \eqref{regular_eps_convergence} follows.

\qed

\vspace{0.7cm}


\section*{Acknowledgements}\thispagestyle{empty}
H. J. Choi was supported by the National Research Foundation of Korea (grant RS-2023-00280065) and Education and Research Promotion program of KOREATECH in 2023.
S. Kim was supported by the National Research Foundation of Korea (grant NRF-2022R1F1A1063379, RS-2023-00217116) and Korea Institute for Advanced Study(KIAS) grant funded by the Korea government (MSIP).
Y. Koh was supported by the National Research Foundation of Korea (grant NRF-2022R1F1A1061968).


\end{document}